\documentclass[twoside,10pt]{article}
 \usepackage{amsfonts,amscd,amssymb}
 \newcommand{\seqn}{\begin{equation}}
 \newcommand{\eeqn}{\end{equation}}
 \newcommand{\seqna}{\begin{eqnarray}}
 \newcommand{\eeqna}{\end{eqnarray}}
 \newenvironment{proof}{{\it Proof.}}{\hfill $ \square $ \vskip 4mm}

\newtheorem{definition}{Definition}[section]
 \newtheorem{lemma}[definition]{Lemma}
 \newtheorem{proposition}[definition]{Proposition}
 \newtheorem{corollary}[definition]{Corollary}
 
 \newtheorem{remarks}[definition]{Remarks}
 \newtheorem{theorem}[definition]{Theorem}

 \let\BBox\Box
 \def\Box{$\BBox$}

 \def\hfl#1#2{\smash{\mathop{\hbox to 10mm{\rightarrowfill}}
 \limits^{\scriptstyle#1}_{\scriptstyle#2}}}

 \def\lettrine#1#2#3{\noindent\hangindent#1\hangsfter-#2
 \hskip-#1\smash{\hbox to#1{#3\hfill}}\ignorespaces}

 \title{\bf SOME HOMOLOGICAL RESULTS IN THE CATEGORY OF COLOUR $(A,H)$-HOPF MODULES}
 \author{{\bf
 Thomas Gu\'ed\'enon}\\ {\it D\'epartement de Math\'ematiques} \\ {\it Universit\'e de Ziguinchor, BP 523 Ziguinchor, S\'en\'egal}\\ thomas.guedenon@univ-zig.sn}
 \date{}
 
\begin{document}
 	
\maketitle
 
\begin{abstract} Let $\Bbbk$ be a field, $H$ a colour Hopf algebra and $A$ a graded $H$-comodule colour algebra. We give a sufficient condition for a colour $(A,H)$-Hopf module to be injective as a graded $H$-comodule and we deduce relative projectivity in the category of colour $(A,H)$-Hopf modules. We generalize the Fundamental Theorem of $(A,H)$-Hopf modules to the context of colour $(A,H)$-Hopf modules. Using this result, we show that the categories of graded $A^{coH}$-modules and of colour $(A,H)$-Hopf modules are equivalent, $A$ is faithfully flat as a graded right $A^{coH}$-module and is a graded Hopf-Galois extension of $A^{coH}$. Under some assumptions, we show that $M^{coH}$ is a graded $A$-module and we prove that the graded global dimension of $A$ is equal to the graded projective dimension of the graded $A$-module $A^{coH}$. 
\end{abstract}
 
 
 
\section{Introduction}

Let $\Bbbk$ be a field and $H$ a Hopf algebra over $\Bbbk$. Let $ A $ be an $H$-comodule algebra. In \cite[Theorem 1]{Doi}, Doi proved that if there exists a right $H$-comodule map $\theta: H\longrightarrow A $ with $ \theta(1_{H})=1_{A}$, then every right $(A,H)$-Hopf module is an injective $H$-comodule. In the same paper, he proved the Fundamental Theorem in the category of right $(A,H)$-Hopf modules \cite[Theorem 3]{Doi}: if there is an $H$-colinear algebra map $ \phi:H\rightarrow A$, then the $ \Bbbk $-linear map
$$M^{coH}\otimes_{B}A\rightarrow  M; m\otimes_{B} a \mapsto ma $$ is an isomorphism of right $(A,H)$-Hopf modules, where $ B=A^{coH} $ is the subalgebra of coinvariants of $A$. A similar proof works for left $(A,H)$-Hopf modules. These results of Doi have been generalized to the category of Doi-Hopf modules in \cite{Gue2}.

Throughout the paper, $G$ is a multiplicative abelian group with a bicharacter  $(\cdot /\cdot):G\otimes G\rightarrow \Bbbk^{\times}$. 
A colour Hopf algebra is a colour algebra which is a colour coalgebra satisfying some compatibility conditions. Its unique difference from a Hopf algebra is that the comultiplication is an algebra homomorphism, not for the componentwise multiplications, but for the twisted multiplication by Lusztig's rule. The notion of colour Hopf algebras first appeared in the book of Montgomery \cite[10.5.11]{Mon}. 
For general facts and further information on colour Hopf algebras, we refer to \cite{CPV, Fel, Gue1, Gue3, Kac, Mon, Sch1, Sch2, Wang}

When $H$ is a colour Hopf algebra, Wang proved in \cite[Theorem 7]{Wang} the Fundamental theorem in the category of colour Hopf modules. He also proved \cite[Theorem 9]{Wang} that the graded global dimension of a colour Hopf algebra coincides with the projective dimension of the graded trivial $H$-module $\Bbbk$: this is a generalization of Lorenz-Lorenz \cite{LL} in the ungraded case. It is natural to ask if one can obtain an analog of this result in the category of colour $(A,H)$-Hopf modules.

Our aim is to generalize the results of Doi and Wang mentioned above to the context of colour $(A,H)$-Hopf modules, where $H$ is a colour Hopf algebra and $A$ is a $G$-graded $H$-comodule colour algebra. A colour $(A,H)$-Hopf module is a graded left $A$-module and a graded left $H$-comodule such that the graded $A$-action and the graded $ H $-coaction are compatible in the  natural way.

$\bullet$  We prove (Theorem $\ref{theorem1}$) that if there exists a graded $H$-colinear map $\theta : H\longrightarrow A $ with $ \theta(1_{H})=1_{A}$ , then every colour $(A,H)$-Hopf module is a gr-injective graded $H$-comodule ({\bf it is a generalization of \cite[Theorem 1]{Doi}}). If furthermore, $A$ and $H$ are colour-commutative and the bicharacter is symmetric, then every colour $(A,H)$-Hopf module which is injective as a graded $A$-module is injective as a colour $(A,H)$-Hopf module. 

For a colour $(A,H)$-Hopf module $M$, denote by $ M^{coH}$ the graded subspace of coinvariants of $M$. Then $B=A^{coH}$ is a colour subalgebra of $A$, $M^{coH}$ is a graded left $ A^{coH} $-module and $ A\otimes_{B} M^{coH}$ is a colour $(A,H)$-Hopf module. 

$\bullet$ We prove (Theorem $\ref{theorem2}$) that if there exists a graded $H$-colinear map $ \theta : H\longrightarrow A $ which is also a homomorphism of colour algebras, then a Fundamental theorem holds for colour $(A,H)$-Hopf modules, that is, the graded $\Bbbk $-linear map $$A\otimes_{B} M^{coH}\rightarrow  M; a\otimes_{B} m\mapsto am $$ is an isomorphism of colour $(A,H)$-Hopf modules ({\bf it is a generalization of \cite[Theorem 3]{Doi} and \cite[Theorem 7: first assertion]{Wang}}).

$\bullet$ When $A$ is colour-commutative colour $H$-simple, we show that $A^{coH}$ is a colour division ring (Lemma $\ref{lemma8}$). We deduce from the Fundamental theorem that every colour $(A,H)$-Hopf module is graded free as a graded $A$-module (Corollary $\ref{corollary4}$) ({\bf it is a generalization of \cite[Theorem 7: second assertion]{Wang}}). 

$\bullet$ We show (Proposition $\ref{proposition1}$) that the functor 
$$F =A\otimes_{B} - :  {}_{\textnormal{gr}-B}\mathcal{M}\longrightarrow {}_{\textnormal{gr}- A}^{\textnormal{gr}-H}\mathcal{M} $$ is left adjoint to the functor 
$$G :{}_{\textnormal{gr}-A}^{\textnormal{gr}-H}\mathcal{M}\longrightarrow {}_{\textnormal{gr}-B}\mathcal{M}; ~~M\mapsto M^{coH}.$$ 
Using this result and the Fundamental Theorem, we prove that  

(i) the functor $G$ is dual Maschke (see \cite[Theorem 3.4]{CM} for the definition), that is, every object of $_{\textnormal{gr}-A}^{\textnormal{gr}-H}\mathcal{M}$ is $G$-relative projective (Corollary $\ref{corollary5}$).

(ii) the adjoint pair $(F,G)$ is a pair of inverses equivalence of categories, $A$ is faithfully flat as a graded right $B$-module and is a graded Hopf-Galois extension of $B$ (Theorem $\ref{theorem3}$).

$\bullet$ We show under some assumptions that for every colour $(A,H)$-Hopf module $M$, the coinvariant module $M^{coH}$ (in particular, $A^{coH}$) is a graded left $A$-module (Proposition $\ref{proposition4}$), and we prove (Theorem $\ref{theorem4}$) that the graded global dimension of $A$ is equal to the graded projective dimension of the graded $A$-module $A^{coH}$ which is also equal to the projective dimension of the $A$-module $A^{coH}$ ({\bf it is a generalization of \cite[Theorem 9]{Wang}}).

By analogy to the ungraded case, we can call a graded $H$-colinear map $\theta: H \rightarrow A$ such that $\theta(1_H)=1_A$ a graded total integral of $A$. 

All the above notions will be defined in the text. All graded vector spaces will be over $ \Bbbk $. Map always means a graded $ \Bbbk $-linear map, and the unadorned tensor product $ V\otimes W $ is understood to be  
$ V\otimes_{\Bbbk} W $. All expressions of linear algebra are given for homogeneous elements and are supposed to be extended to unhomogeneous elements by linearity.

\section{Preliminary results}

In this section, we will introduce the notions of colour algebras, colour coalgebras, graded modules over a colour algebra and graded comodules over a colour coalgebra. 

\begin{definition}
Let $G$ be a multiplicative abelian group with identity element $e$.\\ A bicharacter $(\cdot /\cdot)$ on $G$ is a map from $G\times G$ to $\Bbbk^{\times}$, where $\Bbbk^{\times}$ is the set of invertible elements of  $\Bbbk$ satisfying
\item[i)] $(g / (g'g''))=(g /g')(g /g'') ~~for~~ all ~~g,g',g''\in G;$ 
\item[ii)] $((g g') / g'')=(g /g'')(g' /g'') ~~for~~ all ~~g,g',g''\in G$;
	\\
It is easy to see that $(e/g)=(g/e)=1_{\Bbbk}$ for all $g\in G$. In particular, for a fixed $g\in G$, the induced map $(g/.):G\longrightarrow \Bbbk^{\times}$ defined by $(g/.)(g')=(g/g')$ is a homomorphism of groups. When $(g/g')(g'/g)=1$, we say that the bicharacter is symmetric: in this case, the conditions i) and ii) are equivalent. 
\end{definition}

A graded vector space $M$ is a $\Bbbk$-vector space with the decomposition \\$M=\bigoplus_{g\in G}M_{g}$, where the $M_{g}$ are $\Bbbk$-subspaces of $M$. The elements of $M_g$ are called the homogeneous elements of degree $g$ of $M$ and they will be denoted by $m_g$. The degree of a homogeneous element $m\in M$ will be denoted by $\mid m\mid$. In this paper, all unspecified elements are homogeneous.

Let $M=\bigoplus_{g\in G}M_{g}$ be a graded vector space. A graded vector subspace of $M$ is a $\Bbbk$-vector subspace $N$ of $M$ such that  $$N=\bigoplus_{g\in G}(N \cap M_{g}), \quad \hbox{that is}, \quad N_{g}=N\cap M_{g}.$$

Let $M$ and $N$ be graded vector spaces. We give $M \otimes N$ a $G$-grading by $$(M\otimes N)_{g}=\bigoplus_{g'g''=g}(M_{g'}\otimes N_{g''}).$$  The twist map $\tau: M\otimes N\longrightarrow N\otimes M$ is defined by $$\tau(m\otimes n)=(\mid m\mid /\mid n\mid)(n\otimes m),~~m~~\in ~~M ~~and ~~n~~\in N.$$ 
A graded $\Bbbk$-linear map from $ M $ to $ N $ is a $\Bbbk$-linear map $ f $ from $ M $ to $ N $ such that $ f(M_{g}) \subset N_{g}$. Clearly, $\tau $ is a graded $k$-linear map.

A graded algebra is a graded vector space $A=\bigoplus_{g\in G}A_{g}$ which is an algebra $A$ over $\Bbbk$ (not necessarily associative  with identity) such that $A_{g}A_{g'}\subseteq A_{gg'}$, for all $g,g'\in G$. 
A colour algebra is a graded algebra which is associative with identity $1_A$. We have $1_A \in A_e$. The field $\Bbbk$ is a colour algebra with the trivial gradation. The multiplication and the unit map of $A$ are graded $\Bbbk$-linear maps.

A colour algebra is colour-commutative if $$ab=(\mid a\mid /\mid b\mid)ba \quad \forall \quad a,b \in A.$$

If $A$ and $B$ are colour algebras, a graded $\Bbbk$-linear map $f:A\longrightarrow B$ which is a $\Bbbk$-algebra homomorphism will be called a homomorphism of colour algebras.\\

The colour tensor product algebra of $A$ and $B$ will be the graded $\Bbbk$-vector space $A\otimes B$ equipped with the multiplication defined by $$(a\otimes b)(c\otimes d)=(\mid b \mid /\mid c \mid)(ac)\otimes (bd), ~~\forall~~a,c\in A,~~\forall~~ b,d\in B.$$
A colour subalgebra $B$ of a colour algebra $A=\bigoplus_{g\in G}A_{g}$ is a graded vector subspace $B$ of $A$ which is a subalgebra of $A$.  

\begin{definition} Let $A$ be a colour algebra. A graded left $A$-module $M$ is a $G$-graded vector space $M=\bigoplus_{g\in G}M_{g}$ which is a left $A$-module such that $A_{g}M_{g'}\subseteq M_{gg'}$ for $g,g'\in G$.
\end{definition}

A graded $A$-submodule $N$ of a graded $A$-module $M=\bigoplus_{g\in G}M_{g}$ is a graded vector subspace $N$ of $M$ which is a left $A$-submodule of $M$.

If $M$ is a graded left $A$-module and if $N$ is a graded $A$-submodule of $M$, then $M/ N$ is a graded left $A$-module: the gradation is given by $$(M/ N)_{g}=M_{g}/ (M_{g}\cap N)=(M_{g}+N)/ N.$$

Let $ M $ and $N$ be graded left $A$-modules. A graded $A$-linear map from $ M $ to $N$ is a graded $\Bbbk$-linear map which is  left $A$-linear.
We denote by $_{\textnormal{gr}-A}\mathcal{M}$  the category of graded left  $A$-modules with graded $A$-linear maps, and by $_{\textnormal{gr}-A}Hom(M,N)$ the graded vector space of graded $ A $-linear maps from $M$ to $N$.

Let us recall the following well-known results of graded ring theory \cite{NV1} and \cite{NV2}. 

- Let $N$ be a graded left $A$-module. For every $g$ in $G$, $N(g)$ is the graded $A$-module obtained from
$N$ by a shift of the gradation by $g$. As vector spaces, $N$ and $N(g)$ coincide, and the actions of $A$ on $N$ and $N(g)$ are the same, but the gradations are related by $N(g)_{g'}=N_{gg'}$ for all $g' \in G$.

- An object of $_{\textnormal{gr}-A}\mathcal{M}$ is projective in $_{\textnormal{gr}-A}\mathcal{M}$ if and only if it is projective in $_{A}\mathcal{M}$, the category of $A$-modules.

- By \cite[Chapter 2, Page 21]{NV2}, each $A(g)$ is a projective object in $_{\textnormal{gr}-A}\mathcal{M}$ for $g \in G$.

- An object of $_{\textnormal{gr}-A}\mathcal{M}$ is free in $_{\textnormal{gr}-A}\mathcal{M}$ if it has an $A$-basis consisting of homogeneous elements or equivalently, if it is isomorphic to some $\bigoplus_{i \in I}A(g_i)$, where $I$ is an index set and $(g_i, i \in I)$ is a family of elements of $G$ \cite[Chapter 2, Page 21]{NV2}. If $M$ is a free object in $_{\textnormal{gr}-A}\mathcal{M}$, we say that $M$ is graded free or gr-free. 

- Any object of $_{\textnormal{gr}-A}{\mathcal M}$ is a quotient of a free object in $_{\textnormal{gr}-A}{\mathcal M}$, and any projective object in $_{\textnormal{gr}-A}\mathcal{M}$ is isomorphic to a direct summand of a free object in $_{\textnormal{gr}-A}\mathcal{M}$.

\begin{definition}	A colour coalgebra $C$ is a graded $\Bbbk$-vector space $C=\bigoplus_{g\in G}C_{g}$ which is a coalgebra such that the counit $\varepsilon_{C} :C\longrightarrow \Bbbk$ and the comultiplication $\bigtriangleup_{C} :C\longrightarrow C\otimes C$ are graded $\Bbbk$-linear maps. 

Note that we have: $\bigtriangleup_{C} (C_{g})\subseteq \bigoplus_{h \in G}C_{gh^{-1}}\otimes C_{h}$ and $\varepsilon_{C}( C_{g})=\{0_k\}$ for $g\neq e$.
\end{definition}

\begin{definition} A colour bialgebra $B$ is a datum $(A, m_{B}, \mu_{B}, \bigtriangleup_{B}, \varepsilon_{B})$ such that\\
$B=\bigoplus_{g\in G}B_{g}$ is a colour algebra with multiplication $m_{B}: B\otimes B\longrightarrow B$ and unit map $\mu_{B}: \Bbbk\longrightarrow B$, $(B,\bigtriangleup_{B}, \varepsilon_{B})$ is a colour coalgebra with respect to the same grading.\\
The counit $\varepsilon_{B} :B\longrightarrow\Bbbk$ and the comultiplication $\bigtriangleup_{B} :B\longrightarrow B\otimes B$ are graded algebras maps in the sense that $$\varepsilon_{B} (bb')=\varepsilon_{B} (b)\varepsilon_{B} (b'), \quad \varepsilon_{B} (1_B)=1_{\Bbbk}; ~~b,b'\in B$$ and $$\bigtriangleup_{B} (bb')=(\mid b_{2}\mid /\mid b'_{1}\mid)(b_{1}b'_{1})\otimes (b_{2}b'_{2}), \quad \bigtriangleup_B(1_B)= 1_B \otimes 1_H ;~~b,b'\in B.$$ 
\end{definition}

\begin{definition} A colour Hopf algebra $H$ is a six-tuple $(H, m_{H}, \mu_{H}, \bigtriangleup_{H}, \varepsilon_{H}, {\mathcal S}_{H})$ such that\\
$(H, m_{H}, \mu_{H}, \bigtriangleup_{H}, \varepsilon_H)$ is a colour bialgebra and ${\mathcal S}_{H}:H\longrightarrow H$ is a graded $\Bbbk$-linear map (called the antipode of $H$) such that $$h_{1}{\mathcal S}_{H}(h_{2})=\varepsilon_{H}(h)1_H= {\mathcal S}_{H}(h_{1})h_{2}$$ for all homogeneous elements $h\in H$, where $\bigtriangleup_{H} (h)=h_{1}\otimes h_{2}$. 
\end{definition}

It is easy to see that for a colour Hopf algebra $H$, the antipode $S_H$ preserves the degree, that is $\mid S_H(h) \mid=\mid h \mid$ for all homogeneous element $h \in H$ and satisfies:
$${\mathcal S}_{H}(hh')=(\mid h\mid /\mid h'\mid){\mathcal S}_{H}(h'){\mathcal S}_{H}(h) \quad \forall h,h' \in H$$ and 
$$\bigtriangleup_{H}({\mathcal S}_{H}(h))= (\mid h_{1}\mid /\mid h_{2}\mid) ({\mathcal S}_{H}(h_{2})\otimes {\mathcal S}_{H}(h_{1})) \quad \forall h,h'\in H.$$

Any Hopf algebra is a colour Hopf algebra: take $G=\{e\}$ (so $(e/e)=1_k$).

If $G$ is any abelian group and $(g/g')=1_{\Bbbk}$ for any $g,g' \in G$, then a graded Hopf algebra is a colour Hopf algebra.

The universal enveloping algebra of a Lie colour algebra is a colour-cocommutative colour Hopf algebra.

If $G= \frac{Z}{2Z}$ and $(g+2\mathbb Z/g'+2\mathbb Z)=(-1)^{g+g'}$, a colour algebra is called a superalgebra in \cite{Sch1}, and a colour Hopf algebra is called a Hopf superalgebra in \cite{Fel,Sch1,Sch2} .

\begin{lemma} \label{lemma1} Let $H$ be a colour Hopf algebra. Let $M$ be a graded vector space. If $m \in M$ and $h \in H$ are homogeneous elements, then $(\mid m\mid /\mid h\mid)\varepsilon_{H} (h)=\varepsilon_{H}(h)$.
\end{lemma}

\begin{proof} If $\mid h\mid \neq e$, then $\varepsilon_{H}(h)=0_{\Bbbk}$, so $(\mid m\mid /\mid h\mid)\varepsilon_{H} (h)=0_{\Bbbk}=\varepsilon_{H} (h)$.\\ If $\mid h\mid = e$, then $(\mid m\mid /\mid h\mid)=(\mid m\mid /\mid e\mid)=1$, then $(\mid m\mid /\mid h\mid)\varepsilon_{H} (h)=\varepsilon_{H}(h)$. 
\end{proof}

\begin{definition} Let $H$ be a colour Hopf algebra. A graded left $H$-comodule $M$ is a graded $\Bbbk$-vector space together with a map $\chi_{M}: M\longrightarrow H\otimes M$ defined by $$\chi_M(m_{g})= \sum_{g' \in G}h_{g'^{-1}}\otimes m_{g'g}$$ satisfying the usual left comodule properties:\\
$(Id_{H}\otimes \chi_{M})\circ \chi_{M}=(\Delta_{H}\otimes Id_{M})\circ \chi_{M}$ and $(\varepsilon_{H} \otimes Id_{M})\circ \chi_{M}=Id_{M}$.
\end{definition}

In Sweedler's notation, we write $\chi_{M}(m)=m_{(-1)}\otimes m_{(0)}$ for all homogeneous element $m\in M$ with $\mid m_{(0)}\mid\mid m_{(-1)}\mid=\mid m \mid $, and the graded left $H$-comodule conditions on $M$ are $$m_{(-1)1}\otimes m_{(-1)2}\otimes m_{(0)}=m_{(-1)}\otimes m_{(0)(-1)}\otimes m_{(0)(0)}=m_{(-2)}\otimes m_{(-1)}\otimes m_{(0)},$$ and 
$$  \varepsilon_{H} (m_{(-1)})m_{(0)}=m, ~~\forall~~m\in M.$$\\
Let $H$ be a colour Hopf algebra. If $ M $ and $ N $ are graded left $H$-comodules, a graded $H$-colinear map from $ M $ to $ N $  is a graded $\Bbbk$-linear map which is also a left $H$-colinear map from $ M $ to $ N $, that is, a  graded $\Bbbk$-linear map such that $$f(m)_{(-1)}\otimes f(m)_{(0)}=m_{(-1)}\otimes f(m_{(0)})~~with ~~\chi_{M}(m)=m_{(-1)}\otimes m_{(0)}.$$

\bigskip 

A graded $H$-subcomodule of a graded left $H$-comodule $M=\bigoplus_{g\in G}M_{g}$ is a graded $\Bbbk$-vector subspace $N$ of $M$ which is an $H$-subcomodule of $M$. So for a graded $H$-subcomodule $N$ of $M$, we have $\chi(N_{g})\subseteq H_{g'^{-1}}\otimes N_{g'g}~~\forall~~g,g'\in G.$

Direct sums of graded left $H$-comodules are graded left $H$-comodules. We denote by $^{\textnormal{gr}-H}\mathcal{M}$ the category of graded left $H$-comodules with graded $H$-colinear maps, and by $^{\textnormal{gr}-H}Hom(M,N)$ the graded vector space of all graded $H$-colinear maps from $M$ to $N$. 

Let $M$ be a graded left $H$-comodule. The vector space of the coinvariant elements of $M$ is $$M^{coH}=\{m\in M~~ \hbox{such that} ~~ \chi(m)=1\otimes m \}.$$
Note that $M^{coH}$ is a graded vector subspace of $M$. From now a graded $H$-comodule (graded $A$-module) means a graded left $H$-comodule (graded left $A$-module).

\begin{lemma} \label{lemma2} Let $H$ be a colour Hopf algebra. If $M$ is a graded $H$-comodule, then $H\otimes M$ is a graded $H$-comodule with the coaction defined by: $\chi_{H\otimes M}(h\otimes m)=h_{1}\otimes h_{2}\otimes m$ for all $h\in H$ and $m\in M$.
\end{lemma}

\section{The main results} 
We will establish the main results of the paper. 

\subsection{Injectivity of colour $(A,H)$-Hopf modules}
We begin with the definition of a graded $H$-comodule colour algebra. 

\begin{definition} Let $H$ be a colour Hopf algebra. A graded $H$-comodule colour algebra is a colour algebra $A$ which is a graded left $H$-comodule such that  $$\chi_{A}(aa') =(\mid a_{(0)}\mid /\mid a'_{(-1)} \mid) (a_{(-1)}a'_{(-1)})\otimes (a_{(0)}a'_{(0)}) $$~~and~~ $$\chi_{A}(1_{A})=1_{H}\otimes 1_{A},$$ that is, the multiplication and the unit map are graded $H$-colinear, where the left $H$-coaction on $A \otimes A$ is given by $$(a \otimes a')_{(-1)} \otimes (a \otimes a')_{(0)}=(\mid a_{(0)} \mid /\mid a'_{(-1)} \mid)a_{(-1)}a'_{(-1)} \otimes a_{(0)} \otimes a'_{(0)}$$.
\end{definition}

We can see that $H$ is a graded $H$-comodule colour algebra with $\chi_{H} (h)=\Delta_{H} (h)$, and we have $H^{coH}=1_{H}\Bbbk$.

From now on $H$ is a colour Hopf algebra and $A$ is a graded $H$-comodule colour algebra.

\begin{lemma} \label{lemma3} $A^{coH}$ is a colour subalgebra of $A$ called the colour subalgebra of coinvariants of $A$.
\end{lemma}

\begin{proof} Let $a,a'\in A^{coH}$. Then we have $ a_{(-1)}\otimes a_{(0)}=1_{H}\otimes a~~and~~a'_{(-1)}\otimes a'_{(0)} =1_{H}\otimes a'$. Therefore we have
\begin{eqnarray*}
\chi_{A}(aa') &=&  (\mid a_{(0)}\mid /\mid a'_{(-1)} \mid) (a_{(-1)}a'_{(-1)})\otimes (a_{(0)}a'_{(0)})\\
&=& (\mid a \mid /\mid 1_H \mid) 1_{H}\otimes (aa')\\
&=&  (\mid a \mid /\mid e \mid) 1_{H}\otimes (aa') = 1_{H}\otimes (aa')
\end{eqnarray*}
It is clear that $1_{A} \in A^{coH}$.
\end{proof}

\begin{definition} We say that $M$ is a colour $(A,H)$-Hopf module if $M$ is a graded $A$-module and a graded $H$-comodule such that:
$$\chi_{M} (am)=(\mid a_{(0)} \mid /\mid m_{(-1)} \mid)(a_{(-1)}m_{(-1)})\otimes (a_{(0)}m_{(0)}) ~~\forall ~~a\in A,~~m\in M.$$
\end{definition}

The graded $H$-comodule colour algebra $A$ itself is a colour $(A,H)$-Hopf module.
We denote by $_{\textnormal{gr}-A}^{\textnormal{gr}-H}\mathcal{M}$ the category of colour $(A,H)$-Hopf modules: its morphisms are the graded $A$-linear graded $ H $-colinear maps. A graded $A$-linear graded $H$-colinear map will be called a morphism of colour $(A,H)$-Hopf modules. It is easy to see that $_{\textnormal{gr}-A}^{\textnormal{gr}-H}\mathcal{M}$ is an abelian category. We denote by $_{\textnormal{gr}-A}^{\textnormal{gr}-H}Hom(M,N)$ the graded vector space of all morphisms of colour $(A,H)$-Hopf modules from $M$ to $N$. 

\begin{definition} Let $M$ be a graded $H$-comodule. A graded $H$-comodule $I$ is said to be an injective object in $^{\textnormal{gr}-H}\mathcal{M}$ or a gr-injective graded $H$-comodule if for all injective graded $ H$-colinear map $i: N\longrightarrow P$ and for all graded $ H $-colinear map $ f:N\longrightarrow I$, there is a graded $ H $-colinear map $\overline{f} :P \longrightarrow I$ such that, $\overline{f} \circ i=f$ ($N$ is any left graded $H$-comodule). 
\end{definition}

\begin{lemma} \label{lemma4} 
\begin{enumerate}
\item Let $N$ be a graded vector space. Then $H \otimes N$ is a graded $H$-comodule : the $H$-coaction is $\Delta_{H} \otimes id_N$. In particular, $(1_H \Bbbk) \otimes N$ is a graded $H$-subcomodule of $H \otimes N$.
	
\item If furthermore, $N$ is a graded $A$-module, then $H \otimes N$ is a colour $(A,H)$-Hopf module: the $H$-coaction is $\Delta_{H} \otimes id_N$ and the $A$-action is given by
$$a(h \otimes n)=(\mid a_{(0)} \mid / \mid h \mid)(a_{(-1)}h) \otimes (a_{(0)}n), \quad a \in A, n \in N, h \in H.$$
\end{enumerate}
\end{lemma}

\begin{proof} 1. is well known.

2. We have
$$\begin{array}{l} (aa')(h \otimes n) \\
= (\mid (aa')_{(0)} \mid / \mid h \mid)((aa')_{(-1)}h) \otimes ((aa')_{(0)} n)\\
= (\mid a_{(0)} \mid \mid a'_{(0)} \mid / \mid h \mid)(\mid a_{(0)} \mid / \mid a'_{(-1)}\mid )(a_{(-1)}a'_{(-1)}h) \otimes (a_{(0)}a'_{(0)} n)\\
= (\mid a_{(0)} \mid / \mid h \mid)(\mid a'_{(0)} \mid / \mid h \mid)(\mid a_{(0)} \mid / \mid a'_{(-1)}\mid )(a_{(-1)}a'_{(-1)}h) \otimes (a_{(0)}a'_{(0)} n)\\
= (\mid a'_{(0)} \mid / \mid h \mid)(\mid a_{(0)} \mid / \mid a'_{(-1)}\mid )(\mid a_{(0)} \mid / \mid h \mid)(a_{(-1)}(a'_{(-1)}h)) \otimes (a_{(0)}(a_{(0)} n))\\
=(\mid a'_{(0)} \mid /\mid h\mid)[(\mid a_{(0)}\mid / \mid (a'_{(-1)}h) \mid)[a_{(-1)}(a'_{(-1)}h)] \otimes [a_{(0)}(a'_{(0)}n)] \\
= (\mid a'_{(0)} \mid /\mid h\mid)a[(a'_{(-1)}h) \otimes (a'_{(0)}n)] \\	
= a[a'(h\otimes n)],\end{array}$$
and $H \otimes N$ is a graded $A$-module. We have 

$$\begin{array}{l}(a(h\otimes n))_{(-1)} \otimes  (a(h \otimes n))_{(0)} \\
=(\mid a_{(0)} \mid / \mid h \mid)[(a_{(-1)}h) \otimes (a_{(0)}n)]_{(-1)} \otimes [(a_{(-1)}h) \otimes (a_{(0)}n)]_{(0)}\\
=(\mid a_{(0)} \mid / \mid h \mid) [(a_{(-1)}h)_{1} \otimes (a_{(-1)}h)_{2} \otimes (a_{(0)}n)] \\
=(\mid a_{(0)} \mid / \mid h_1 \mid \mid h_2 \mid ) (\mid a_{(-1)2} \mid / \mid h_1 \mid )[(a_{(-1)1}h_{1})\otimes (a_{(-1)2}h_{2}) \otimes (a_{(0)}n)]\\
=(\mid a_{(0)} \mid / \mid h_1\mid \mid h_2 \mid) (\mid a_{(-1)} \mid / \mid h_1 \mid )[(a_{(-2)}h_{1}) \otimes (a_{(-1)}h_{2}) \otimes(a_{(0)}n)]\\ 
=(\mid a_{(0)} \mid / \mid h_1 \mid)(\mid a_{(0)}\mid / \mid h_2 \mid) (\mid a_{(-1)} \mid / \mid h_1 \mid )[(a_{(-2)}h_{1})\otimes (a_{(-1)}h_{2}) \otimes (a_{(0)}n)]\\ 
=(\mid a_{(0)} \mid \mid a_{(-1)} \mid  / \mid h_1\mid )(\mid a_{(0)} \mid / \mid h_2\mid )(a_{(-2)}h_1) \otimes (a_{(-1)}h_2) \otimes (a_{(0)}n) \\
= (\mid a_{(0)(0)} \mid \mid a_{(0)(-1)} \mid  / \mid h_1\mid )(\mid a_{(0)(0)} \mid / \mid h_2\mid )(a_{(-1)}h_1) \otimes (a_{(0)(-1)}h_2) \otimes (a_{(0)(0)}n) \\
= (\mid a_{(0)} \mid / \mid h_1\mid )(a_{(-1)}h_1) \otimes (a_{(0)}(h_2 \otimes n))\\
=(\mid a_{(0)} \mid / \mid (h \otimes n)_{(-1)} \mid )(a_{(-1)}(h \otimes n)_{(-1)}) \otimes (a_{(0)}(h \otimes n)_{(0)}) \end{array}$$
Thus $H \otimes N$ is a colour $(A,H)$-Hopf module.
\end{proof}

From Lemma $\ref{lemma4}$, $H \otimes A$  is a colour $(A,H)$-Hopf module : the $H$-coaction is $\Delta_H \otimes id_A $, and the $A$-action is given by
$$a(h \otimes a')=(\mid a_{(0)}\mid / \mid h \mid)(a_{(-1)}h) \otimes (a_{(0)}a'), \quad a,a' \in A, h \in H.$$
In particular, $H \otimes A$ is a graded $A^{coH}$-module with $(1_H \Bbbk) \otimes A$ as a graded $A^{coH}$-submodule. Note that $H \otimes A$ is also a graded $H$-comodule colour algebra.
 
\begin{lemma} \label{lemma5} 
\begin{enumerate}
\item Let $M$ be a graded $H$-comodule and $N$ a graded vector space. Then there is a graded $\Bbbk$-linear isomorphism
$$ \gamma : {}^{\textnormal{gr}-H}Hom(M, H \otimes N) \rightarrow {}_{gr-{\Bbbk}}Hom(M,N)$$ for all $M \in {}^{\textnormal{gr}-H}{\mathcal M}$ and 
$N \in {}_{gr-{\Bbbk}}{\mathcal M}$. This map $\gamma$ is defined by $\gamma(f)=(\varepsilon_H \otimes id_N) \circ f$. The inverse $\gamma'$ of $\gamma$ is given by $\gamma'(g)=(id_H \otimes g) \circ \rho_M$. 
	
\item Let $M$ be a colour $(A,H)$-Hopf module and $N$ a graded $A$-module. Then there is a graded $\Bbbk$-linear isomorphism
$$ \gamma : {}_{\textnormal{gr}-A}^{\textnormal{gr}-H}Hom(M, H \otimes N) \rightarrow {}_{\textnormal{gr}-A}Hom(M,N)$$ for all $M \in {}_{\textnormal{gr}-A}^{\textnormal{gr}-H}{\mathcal M}$ and $N \in {}_{\textnormal{gr}-A}{\mathcal M}$. This map $\gamma$ and its inverse $\gamma'$ are defined as in 1.
\end{enumerate}
\end{lemma}

\begin{proof} 1. By Lemma $\ref{lemma4}$, $H \otimes N$ is a graded $H$-comodule. Choose $f \in {}^{\textnormal{gr}-H}Hom(M, H \otimes N)$ and $m \in M$. Set $f(m)=\sum_{i \in I}(h_i \otimes n_i)$ for some $n_i \in N$ and $h_i \in H$, where $I$ a family set of indexes. Then $\gamma(f)(m)=\sum_{i \in I}(\varepsilon_H(h_i)n_i)$. For $a \in A$, let $g \in {}_{gr-k}Hom(M, N)$. We have
\begin{eqnarray*} \gamma'(g)(m)_{(-1)} \otimes \gamma'(g)(m)_{(0)} &=& (m_{(-1)} \otimes g(m_{(0)}))_{(-1)} \otimes (m_{(-1)} \otimes g(m_{(0)}))_{(0)}\\
&=& m_{(-1)1} \otimes m_{(-1)2} \otimes g(m_{(0)}) \\
&=& m_{(-1)} \otimes m_{(0)(-1)} \otimes g(m_{(0)(0)}) \\
&=&m_{(-1)} \otimes \gamma'(g)(m_{(0)}) \\.\end{eqnarray*}
So $\gamma'(g)$ is $H$-colinear. We have
\begin{eqnarray*} [(\gamma \circ \gamma')(g)](m)&=&  [\gamma(\gamma'(g))](m) \\
&=& (\varepsilon_H \otimes id_N)[\gamma'(g)(m)]\\
&=& (\varepsilon_H \otimes id_N)[(id_H \otimes g)(m_{(-1)} \otimes m_{(0)})]\\
&=& (\varepsilon_H \otimes id_N)[m_{(-1)} \otimes g(m_{(0)})]\\
&=& \varepsilon_H(m_{(-1)})g(m_{(0)})=g(m).\end{eqnarray*} So $\gamma \circ \gamma'=id_{_{gr-k}Hom(M,N)}$. We remark that
$$f(m)_{(-1)} \otimes f(m)_{(0)}=\sum_{i \in I}( h_{i1} \otimes h_{i2} \otimes n_{i}).$$ Applying this formula, we obtain
$$ \begin{array}{rcl}[(\gamma' \circ \gamma)(f)](m) &=& [\gamma'(\gamma(f))](m)\\
&=& [id_H \otimes \gamma(f)](m_{(-1)} \otimes m_{(0)}) \\
&=& [m_{(-1)} \otimes \gamma(f)(m_{(0)})]\\
&=& [m_{(-1)} \otimes (\varepsilon_H \otimes id_N)(f(m_{(0)}))]\\
&=& f(m)_{(-1)} \otimes [(\varepsilon_H \otimes id_N)(f(m)_{(0)})]\\
&=& [\sum_{i \in I}(h_i \otimes n_i)]_{(-1)} \otimes (\varepsilon_H \otimes id_N)[\sum_{i \in I}(h_i \otimes n_i)]_{(0)})\\
&=& \sum_{i \in I}h_{i1} \otimes (\varepsilon_H \otimes id_N)(h_{i2} \otimes n_i)\\
&=& \sum_{i \in I}h_{i1} \otimes (\varepsilon_H(h_{i2})n_i)\\
&=& \sum_{i \in I}[(h_{i1}\varepsilon_H(h_{i2})) \otimes n_i)]\\
&=& \sum_{i \in I} (h_{i} \otimes n_{i}) =f(m).\end{array}$$ So $\gamma' \circ \gamma=id_{^{\textnormal{gr}-H}Hom(M,H \otimes N)}$. The proof of (1) is finished.

2. By Lemma $\ref{lemma4}$, $H \otimes N$ is a colour $(A,H)$-Hopf module. One easily check that $\gamma(f)$ and $\gamma'(g)$ are graded $A$-linear maps. 
\end{proof} 

\begin{corollary} \label{corollary1} If $N$ is a graded vector space (resp., an injective graded $A$-module), then $H \otimes N$ is a gr-injective graded $H$-comodule (resp., an injective colour $(A,H)$-Hopf module).
\end{corollary}

If $H$ is colour-commutative and the bicharacter is symmetric, then we have $${\mathcal S}_H(hh')={\mathcal S}_H(h){\mathcal S}_H(h')\quad \forall h, h' \in H.$$ 

The following formula will be useful for Theorem $\ref{theorem1}$, Lemma $\ref{lemma7}$ and Proposition $\ref{proposition4}$. 

$$m_{(-1)1} \otimes m_{(-1)2} \otimes m_{(0)(-1)} \otimes m_{(0)(0)}= m_{(-2)} \otimes m_{(-1)1} \otimes m_{(-1)2} \otimes m_{(0)}.$$

The following theorem is a generalization of \cite[Theorem 1]{Doi}.

\begin{theorem} \label{theorem1} 
	
Assume that there is a graded $H$-colinear map $\theta: H \rightarrow A$ such that $\theta(1_H)=1_A$.
\begin{enumerate} 
\item Then every colour $(A,H)$-Hopf module is a \textnormal{gr}-injective graded $H$-comodule.
	
\item If $A$ and $H$ are colour-commutative and the bicharacter is symmetric, then every colour $(A,H)$-Hopf module which is injective as a graded $A$-module is injective as a colour $(A,H)$-Hopf module.
\end{enumerate}
\end{theorem}

\begin{proof} 1. Let $M$ be a colour $(A,H)$-Hopf module. The graded  comodule structure map $ \chi_{M} $ is a graded colinear map. We know that $H\otimes M$ is a gr-injective graded $H$-comodule via $ \Delta_{H}\otimes id_{M} $ (Corollary $\ref{corollary1}$). Let us consider the $ \Bbbk $-linear map $ \lambda:H\otimes M\rightarrow M$ defined by $$\lambda(h\otimes m)=\theta(h {\mathcal S}_{H}(m_{(-1)}))m_{(0)}~~\forall~~m\in M,~~ h\in H.$$ Clearly, $\lambda$ is a graded $k$-linear map. We have 
\begin{eqnarray*}
(\lambda \circ  \chi_{M})(m) &=&\lambda(\chi_{M}(m))=\lambda(m_{(-1)}\otimes m_{(0)})\\
&=&\theta(m_{(-1)}{\mathcal S}_{H}(m_{(0)(-1)}))m_{(0)(0)}\\
&=&\theta(m_{(-1)1}{\mathcal S}_{H}(m_{(-1)2}))m_{(0)}\\
&=&\theta(\varepsilon_H(m_{(-1)})1_{H})m_{(0)}\\
&=&\theta(1_{H})(\varepsilon_H(m_{(-1)})m_{(0)}\\
&=&\theta(1_{H})m=1_{A}m=m.
\end{eqnarray*}
Therefore $\lambda \circ \chi_{M}=id_{M}$. We also have 
$$\begin{array}{l}\lambda(h \otimes m)_{(-1)} \otimes  \lambda(h \otimes m)_{(0)} \\
=(\theta(h {\mathcal S}_{H}(m_{(-1)}))m_{(0)})_{(-1)} \otimes (\theta(h {\mathcal S}_{H}(m_{(-1)}))m_{(0)})_{(0)} \\
=(\mid \theta(h {\mathcal S}_{H}(m_{(-1)}))_{(0)}\mid / \mid m_{(0)(-1)}\mid)\theta(h {\mathcal S}_{H}(m_{(-1)}))_{(-1)}m_{(0)(-1)} \otimes \theta(h {\mathcal S}_{H}(m_{(-1)}))_{(0)}m_{(0)(0)} \\
=(\mid \theta((h {\mathcal S}_{H}(m_{(-1)}))_{(0)})\mid / \mid m_{(0)(-1)}\mid)(h {\mathcal S}_{H}(m_{(-1)}))_{(-1)}m_{(0)(-1)} \otimes \theta((h {\mathcal S}_{H}(m_{(-1)}))_{(0)})m_{(0)(0)} \\
=(\mid \theta(h_2 {\mathcal S}_{H}(m_{(-1)})_{2})\mid / \mid m_{(0)(-1)}\mid)(\mid h_2\mid / \mid {\mathcal S}_{H}(m_{(-1)})_{1}\mid) \\
(h_1 {\mathcal S}_{H}(m_{(-1)})_{1}m_{(0)(-1)} \otimes \theta(h_2 {\mathcal S}_{H}(m_{(-1)})_{2})m_{(0)(0)} \\
=(\mid h_2 {\mathcal S}_{H}(m_{(-1)1})\mid / \mid m_{(0)(-1)}\mid)(\mid h_2\mid / \mid m_{(-1)2})\mid)(\mid m_{(-1)1}\mid / \mid m_{(-1)2}\mid) \\
(h_1 {\mathcal S}_{H}(m_{(-1)2})m_{(0)(-1)} \otimes \theta(h_2 {\mathcal S}_{H}(m_{(-1)1}))m_{(0)(0)} \\
=(\mid h_2 \mid / \mid m_{(-1)2}\mid)(\mid m_{(-2)} \mid / \mid m_{(-1)2}\mid)(\mid h_2\mid / \mid m_{(-1)1})\mid)(\mid m_{(-2)}\mid / \mid m_{(-1)1}\mid) \\
(h_1 {\mathcal S}_{H}(m_{(-1)1})m_{(-1)2} \otimes \theta(h_2 {\mathcal S}_{H}(m_{(-2)}))m_{(0)} \\
=(\mid h_2 \mid / \mid m_{(-1)}\mid)(\mid m_{(-2)} \mid / \mid m_{(-1)}\mid)(h_1\varepsilon_{H}(m_{(-1)})  \otimes \theta(h_2 {\mathcal S}_{H}(m_{(-2)}))m_{(0)} \\
=\varepsilon_{H}(m_{(-1)})(h_1 \otimes \theta(h_2 {\mathcal S}_{H}(m_{(-2)})))m_{(0)} \\
=h_1\otimes \theta(h_2 {\mathcal S}_{H}(m_{(-1)}))m_{(0)} \\
=h_1 \otimes \lambda(h_2 \otimes m) \\
=(h \otimes m)_{(-1)}  \otimes \lambda((h \otimes m)_{(0)}). \end{array} $$
The sixth equality is true since ${\mathcal S}_H$ is homogeneous of degree $e$ and because of the formula mentioned before the theorem. The eighth equality uses Lemma $\ref{lemma1}$.
Thus $\lambda$ is $H$-colinear. Since $\chi_{M}: M\longrightarrow H\otimes M$ is an injective graded $H$-colinear map, $ M $ is a direct summand of $H\otimes M$ as a graded $H$-comodule. We deduce that $M$ is a $gr$-injective graded $H$-comodule, being a direct summand of the gr-injective graded $H$-comodule $H \otimes M$. 

2. Let $M$ be a colour $(A,H)$-Hopf module which is injective as a graded $A$-module. We have 
$$ \begin{array}{l}\lambda(a(h \otimes m)) \\
=(\mid a_{(0)}\mid / \mid h \mid)\lambda[(a_{(-1)}h) \otimes (a_{(0)}m)] \\
=(\mid a_{(0)}\mid / \mid h \mid)\theta[a_{(-1)}h{\mathcal S}_H((a_{(0)}m)_{(-1)})](a_{(0)}m)_{(0)} \\
=(\mid a_{(0)(0)}\mid \mid a_{(0)(-1)} \mid / \mid h \mid)(\mid a_{(0)(0)} \mid / \mid m_{(-1)} \mid)\theta[a_{(-1)}h{\mathcal S}_H(a_{(0)(-1)}m_{(-1)})](a_{(0)(0)}m_{(0)}) \\
=(\mid a_{(0)}\mid \mid a_{(-1)} \mid / \mid h \mid)(\mid a_{(0)} \mid / \mid m_{(-1)} \mid)\theta[a_{(-2)}h{\mathcal S}_H(a_{(-1)}m_{(-1)})](a_{(0)}m_{(0)}) \\
=(\mid a_{(0)}\mid \mid a_{(-1)} \mid / \mid h \mid)(\mid a_{(0)} \mid / \mid m_{(-1)} \mid)\theta[a_{(-2)}h{\mathcal S}_H(a_{(-1)}){\mathcal S}_H(m_{(-1)})](a_{(0)}m_{(0)}) \\ 
=(\mid a_{(0)}\mid \mid a_{(-1)} \mid / \mid h \mid)(\mid a_{(0)} \mid / \mid m_{(-1)} \mid)(\mid h \mid / \mid a_{(-1)} \mid ) \theta[a_{(-2)}{\mathcal S}_H(a_{(-1)})h{\mathcal S}_H(m_{(-1)})](a_{(0)}m_{(0)}) \\ 
=(\mid a_{(0)}\mid / \mid h \mid)(\mid a_{(0)} \mid / \mid m_{(-1)} \mid)\theta[a_{(-2)}{\mathcal S}_H(a_{(-1)})h{\mathcal S}_H(m_{(-1)})](a_{(0)}m_{(0)}) \\
=(\mid a_{(0)}\mid / \mid h \mid)(\mid a_{(0)} \mid / \mid m_{(-1)} \mid)\theta[\varepsilon_H(a_{(-1)})h{\mathcal S}_H(m_{(-1)})](a_{(0)}m_{(0)}) \\
=(\mid a\mid / \mid h \mid)(\mid a \mid / \mid m_{(-1)} \mid)\theta[h{\mathcal S}_H(m_{(-1)})](am_{(0)}) \\
=(\mid a \mid / \mid h \mid)(\mid a \mid / \mid{\mathcal S}_H(m_{(-1)}) \mid)\theta(h{\mathcal S}_H(m_{(-1)}))(am_0) \\
=(\mid a \mid / \mid h{\mathcal S}_H(m_{(-1)}) \mid )\theta(h{\mathcal S}_H(m_{(-1)}))(am_0) \\
=(\mid a \mid / \mid \theta(h{\mathcal S}_H(m_{(-1)})) \mid )\theta(h{\mathcal S}_H(m_{(-1)}))(am_0) \\
=(\mid a \mid / \mid \theta(h{\mathcal S}_H(m_{(-1)}) \mid)(\mid \theta(h{\mathcal S}_H(m_{(-1)})) \mid / \mid a \mid)a\theta(h{\mathcal S}_H(m_{(-1)}))m_0 \\
= a[\theta(h{\mathcal S}_H(m_{(-1)}))m_0] \\
=a\lambda (h \otimes m). \end{array}$$
The fifth equality and the sixth equality are true since $H$ is colour-commutative, the bicharacter is symmetric and ${\mathcal S}_H$ is homogeneous of degree $e$. The seventh equality is true since the bicharacter is symmetric. The tenth equality is true since ${\mathcal S}_H$ is homogeneous of degree $e$. The twelfth equality is true since ${\mathcal S}_H$ and $\theta$ are homogeneous of degree $e$. The thirteenth equality is true since $A$ is colour-commutative. It follows that
$$\lambda(a(h \otimes m))=a \lambda(h \otimes m),$$ that is $\lambda$ is graded $A$-linear. Thus $\lambda$ is a homomorphism of colour $(A,H)$-Hopf modules. By Corollary $\ref{corollary1}$, $H\otimes M$ is an injective colour $(A,H)$-Hopf module. It is easy to see that the comodule structure map $\chi_{M}$ is graded $A$-linear for the given $A$-action on $H \otimes M$. Thus, $\chi_M$ is a homomorphism of colour $(A,H)$-Hopf modules. We also know that $\chi_{M}: M \rightarrow H\otimes M$ is an injective map. So $ M $ is a direct summand of $H\otimes M$ as a colour $(A,H)$-Hopf module. It follows that $M$ is an injective colour $(A,H)$-Hopf module, being a direct summand of the injective colour $(A,H)$-Hopf module $H \otimes M$ (Corollary $\ref{corollary1}$).
\end{proof}

Let $M$ be a graded left $H$-comodule. Then $A \otimes M$ is a colour $(A,H)$-Hopf module: the $A$-action is natural while the $H$-coaction is the diagonal coaction, that is,
$$\chi_{A \otimes M}(a \otimes m)= (\mid a_{(0)} \mid /\mid m_{(-1)} \mid)(a_{(-1)}m_{(-1)}) \otimes a_{(0)} \otimes m_{(0)}   \quad \forall a \in A, m \in M.$$
Clearly, the tensor product functor 
$$A \otimes - : {}^{\textnormal{gr}-H}{\mathcal M} \rightarrow {}_{\textnormal{gr}-A}^{\textnormal{gr}-H}{\mathcal M}; M \mapsto A \otimes M$$ is left adjoint to the forgetful functor
$${\mathcal F} : {}_{\textnormal{gr}-A}^{\textnormal{gr}-H}{\mathcal M} \rightarrow {}^{\textnormal{gr}-H}{\mathcal M}, \quad \hbox{that is},$$ for all $M \in {}_{\textnormal{gr}-A}^{\textnormal{gr}-H}{\mathcal M}$ and $N \in {}^{\textnormal{gr}-H}{\mathcal M}$, we have a graded $\Bbbk$-linear isomorphism
$$ \gamma : {}^{\textnormal{gr}-H}Hom(N,{\mathcal F}(M)) \rightarrow {}_{\textnormal{gr}-A}^{\textnormal{gr}-H}Hom(A \otimes N,M).$$ 
This map $\gamma$ is defined by $$\gamma(f)(a \otimes n)=af(n) \quad \forall a \in A, n \in N.$$
The inverse $\gamma'$ of $\gamma$ is given by
$$\gamma'(g)(n)=g(1_A \otimes n).$$ 

Thus if $M$ is a colour $(A,H)$-Hopf module, ${\mathcal F}(M) \in {}^{\textnormal{gr}-H}{\mathcal M}$ is an injective object if and only if $_{\textnormal{gr}-A}^{\textnormal{gr}-H}Hom(A \otimes -,M)$ is an exact functor. 
We now give one consequence of the first part of Theorem $ \ref{theorem1} $ which we can consider as the first step toward a future project in studying injective classes in the category $_{\textnormal{gr}-A}^{\textnormal{gr}-H}\mathcal{M}$. We refer to \cite{EM} for more information about relative homological algebra. For the definition of a sequence in an abelian category, we refer to \cite[Page 3]{EM}. Under the notations above, we have

\begin{corollary} \label{corollary2} Assume that there is a graded $H$-colinear map $\theta :H\rightarrow A$ such that $\theta (1_{H})=1_{A}$. Then the class of colour $(A,H)$-Hopf modules which are injective with respect to the class of sequences of colour $(A,H)$-Hopf modules $ \mathcal{P}=A\otimes \mathcal{P'}$ ($\mathcal{P'}$ is any class of exact sequences of graded $H$-comodules) is the whole category $_{\textnormal{gr}-A}^{\textnormal{gr}-H}\mathcal{M}$ of colour $(A,H)$-Hopf modules. That is, in the usual notations of relative homological algebra, we have $_{\textnormal{gr}-A}^{\textnormal{gr}-H}\mathcal{M}\Rightarrow \mathcal{P}$.
\end{corollary}

\begin{proof} Since $ \Bbbk$ is a field, the functor $A \otimes -$ is exact. Let $Q$ be an exact sequence of graded $H$-comodules. Then $A \otimes Q$ is an exact sequence of $_{\textnormal{gr}-A}^{\textnormal{gr}-H}\mathcal{M}$. Let $E$ be an object of $_{\textnormal{gr}-A}^{\textnormal{gr}-H}\mathcal{M}$. Then  $_{\textnormal{gr}-A}^{\textnormal{gr}-H}Hom(A \otimes Q,E)$ a sequence of $_{gr-{\Bbbk}}\mathcal{M}$. By Theorem $\ref{theorem1}$, $\mathcal{F}(E)$ is a gr-injective graded $H$-comodule. We deduce that if  $\mathcal{P'}$ is any class of exact sequences of $^{\textnormal{gr}-H}\mathcal{M}$, then  $\mathcal{F}(E)$ is injective with respect to $\mathcal {P'}$. It follows that $E$ is injective with respect to the exact sequence $\mathcal {P}=A \otimes {\mathcal P'}$. Thus $_{\textnormal{gr}-A}^{\textnormal{gr}-H}Hom(\mathcal{P},E)$ is an exact sequence of $_{\textnormal{gr}-\Bbbk}\mathcal{M}$. Let us denote by $\mathcal{E}$ the class of all objects of $_{\textnormal{gr}-A}^{\textnormal{gr}-H}\mathcal{M}$ which are injective with respect to ${\mathcal{P}}=A \otimes {\mathcal{P'}}$. Then we have ${\mathcal E}={}_{\textnormal{gr}-A}^{\textnormal{gr}-H}\mathcal{M}$.   
\end{proof}

The following corollary is a generalization of \cite[Corollary Page 246]{Doi}.

\begin{corollary} \label{corollary3} The following assertions are equivalent:
\begin{enumerate}
\item $A$ is a \textnormal{gr}-injective graded $H$-comodule.
\item There is a graded $H$-colinear map $\theta :H\rightarrow A$ such that  $\theta (1_H)=1_A$.
\end{enumerate}
\end{corollary}

\begin{proof} We know that $A$ is an $H$-comodule colour algebra. The unit map $\mu_{H} :\Bbbk \longrightarrow H$ is an injective morphism of graded $H$-comodules. Since $A$ is a gr-injective graded $H$-comodule and the unit map $\mu_{A} :\Bbbk \longrightarrow A$ is a morphism of graded $H$-comodules, there is a graded $H$-colinear map $\theta :H  \longrightarrow A$ such that $\theta \circ \mu_{H} = \mu_{A}$. It follows that $\theta(1_{H})=1_{A}$  because $ \mu_A (1_{\Bbbk})=1_{A}$ and $ \mu_H (1_{\Bbbk})=1_{H}$. Then $ 1. \Rightarrow 2.$. The assertion $2. \Rightarrow 1.$ follows from Theorem $ \ref{theorem1}$.
\end{proof}

\subsection{The fundamental theorem for colour $(A,H)$-Hopf modules}

\begin{lemma} \label{lemma6} Set $B=A^{coH}$.
\begin{enumerate}
		
\item Let $M$ be a colour $(A,H)$-Hopf module. Then $M^{coH}$ is a graded left $B$-module.
		
\item Let $M$ be a graded left $B$-module. Then $A\otimes_{B} M$ is a colour $(A,H)$-Hopf module: the $A$-action is the natural one and the $H $-comodule structure map is given by $\chi (a \otimes_{B} m) = a_{(-1)} \otimes a_{(0)} \otimes_{B} m$.
		
\item Let $M$ be a colour $(A,H)$-Hopf module. Then $ A\otimes_{B} M^{coH}\in {}_{\textnormal{gr}-A}^{\textnormal{gr}-H}{\mathcal M}$ : $a(a' \otimes_{B} m)=(aa') \otimes_{B}  m$ and $\chi (a \otimes_{B}  m) = a_{(-1)} \otimes a_{(0)} \otimes_{B} m$.
		
\end{enumerate}
\end{lemma}

\begin{proof} 1. It is easy.

2. Clearly, $A\otimes_{B}M$ is a graded $ A $-module and a graded left $ H$-comodule. It remains to show the compatibility condition. Let $a,a' \in A$ and $m \in M$. We have
\begin{eqnarray*}
\chi(a(a' \otimes_{B} m))&=& \chi(aa'  \otimes_{B} m)\\
&=& (aa'  \otimes_{B} m)_{(-1)} \otimes (aa' \otimes_{B} m)_{(0)}\\
&=& (aa')_{(-1)} \otimes (aa')_{(0)} \otimes_{B} m\\
&=&(\mid a_{(0)} \mid / \mid a'_{(-1)} \mid)(a_{(-1)}a'_{(-1)}) \otimes (a_{(0)} a'_{(0)})  \otimes_{B} m\\
&=&(\mid a_{(0)} \mid / \mid a'_{(-1)} \mid)(a_{(-1)}a'_{(-1)})\otimes a_{(0)} (a'_{(0)}  \otimes_{B} m)\\
&=&(\mid a_{(0)} \mid / \mid (a' \otimes_Bm)_{(-1)} \mid)a_{(-1)}(a'  \otimes_{B} m)_{(-1)} \otimes a_{(0)}(a'  \otimes_{B} m)_{(0)}
\end{eqnarray*}
3. By 1., $ M^{coH}$ is a graded left $ B$-module. By 2., $ A\otimes_{B} M^{coH}$ is a colour $(A,H)$-Hopf module with the given structures.
\end{proof}

By Lemma $\ref{lemma6}$, $A \otimes_BA$ is a colour $(A,H)$-Hopf module: the $A$-action is $a(a' \otimes_Ba'')=(aa') \otimes_Ba''$ and the $H$-coaction is $\chi(a \otimes_Ba')=a_{(-1)} \otimes ( a_{(0)} \otimes_B a')$. Note that $A \otimes_BA$ is also a graded $H$-comodule colour algebra.

\bigskip

Assume that there is a graded $H$-colinear map $\theta : H \rightarrow A$. Let $M$ be a colour $(A,H)$-Hopf module. Let us consider the graded $ \Bbbk $-linear map $p_M:M\rightarrow M$ defined by $p_M(m)= \theta ({\mathcal S}_{H}(m_{(-1)})) m_{(0)}$ for all $m \in M$. 

\begin{lemma} \label{lemma7}. Assume that there is a graded $H$-colinear map $\theta : H \rightarrow A$ such that $\theta(1_H)=1_A$.
\begin{enumerate}
\item $p_M(M)=M^{coH}$
		
\item $p_M \circ p_M=p_M$: thus $p_M$ is a projection map.

\end{enumerate}
\end{lemma}

\begin{proof} 1. Let $m \in M$. We have

$$\begin{array}{l}\chi_{M}(p_M(m)) \\
=\chi_{M}(\theta({\mathcal S}_{H}(m_{(-1)}))m_{(0)})\\
= (\theta ({\mathcal S}_{H}(m_{(-1)})) m_{(0)})_{(-1)} \otimes (\theta ({\mathcal S}_{H}(m_{(-1)})) m_{(0)})_{(0)}\\
=(\mid \theta ({\mathcal S}_{H}(m_{(-1)}))_{(0)} \mid/ \mid m_{(0)(-1)} \mid) [(\theta ({\mathcal S}_{H}(m_{(-1)}))_{(-1)} m_{(0)(-1)}] \otimes [(\theta ({\mathcal S}_{H}(m_{(-1)}))_{(0)} m_{(0)(0)}]\\	
=(\mid \theta ({\mathcal S}_{H}(m_{(-1)})_{2}) / \mid m_{(0)(-1)}\mid) [{\mathcal S}_{H}(m_{(-1)})_{1} m_{(0)(-1)}] \otimes [\theta ({\mathcal S}_{H}(m_{(-1)})_{2}) m_{(0)(0)}]\\
=(\mid\theta ({\mathcal S}_{H}(m_{(-1)1}))\mid / \mid m_{(0)(-1)} \mid)(\mid m_{(-1)1} \mid / \mid m_{(-1)2} \mid)[{\mathcal S}_{H}(m_{(-1)2}) m_{(0)(-1)}]\\
\otimes [\theta ({\mathcal S}_{H}(m_{(-1)1})) m_{(0)(0)}]\\
=(\mid m_{(-1)1} \mid / \mid m_{(0)(-1)} \mid)(\mid m_{(-1)1} \mid / \mid m_{(-1)2}\mid) [{\mathcal S}_{H}(m_{(-1)2}) m_{(0)(-1)}] \otimes [\theta ({\mathcal S}_{H}(m_{(-1)1})) m_{(0)(0)}]\\
=(\mid m_{(-2)} \mid / \mid m_{(-1)2} \mid)(\mid m_{(-2)} \mid / \mid m_{(-1)1}\mid) [{\mathcal S}_{H}(m_{(-1)1}) m_{(-1)2}] \otimes [\theta ({\mathcal S}_{H}(m_{(-2)})) m_{(0)}]\\
=(\mid m_{(-2)} \mid/ \mid m_{(-1)} \mid) (\varepsilon_{H}(m_{(-1)})1_H) \otimes [\theta ({\mathcal S}_{H}(m_{-2})) m_{(0)}]\\
=(\varepsilon_{H}(m_{(-1)})1_H \otimes [\theta ({\mathcal S}_{H}(m_{-2})) m_{(0)}] \quad  {\bf by \quad Lemma \quad \ref{lemma1}}\\
=1_H \otimes [\theta({\mathcal S}_{H}(m_{(-1)}))m_{(0)}]\\
=1_H \otimes p_M(m).
\end{array}$$
The sixth equality is true since $\mid \theta \mid=\mid {\mathcal S}_H \mid=e$. The seventh equality uses the formula mentioned before Theorem $\ref{theorem1}$.

So $p_M(M)$ is a vector subspace of $ M^{coH}$. We have $m=p_M(m)$ for all $m \in M^{coH}$. So $M^{coH}$ is contained in $p_M(M)$.

2. Let $m \in M$. We have $p_M(m) \in M^{coH}$. Thus we get $p_M(p_M(m))=p_M(m)$. 

\end{proof}
 
We are now in the position to provide the Fundamental Theorem for colour $(A,H)$-Hopf modules: it is a generalization of \cite[Theorem 3]{Doi} and \cite[Theorem 7: first assertion]{Wang}.

\begin{theorem} \label{theorem2}
Let $H$ be a colour Hopf algebra and $A$ a graded $H$-comodule colour algebra. Set $ B= A^{coH}$. Assume that there is a graded $H$-colinear map $\theta :H\rightarrow A$ which is a homomorphism of colour algebras. Then for every colour $(A,H)$-Hopf module $M$, the map $\alpha:A\otimes_{B}M^{coH}\rightarrow M$ defined by $\alpha(a \otimes_{B} m)=am$ is an isomorphism of colour $(A,H)$-Hopf modules.
\end{theorem}

\begin{proof} Clearly, $\alpha$ is a well-defined graded $A$-linear map. For $m\in  M^{coH}$ and $a\in A$, we have 
\begin{eqnarray*}
\alpha (a \otimes_{ B} m)_{(-1)} \otimes \alpha (a \otimes_{ B} m)_{(0)} &=& (am)_{(-1)} \otimes (am)_{(0)}\\
&=& (\mid a_{(0)}\mid /\mid 1_{H} \mid)(a_{(-1)}1_{H})\otimes (a_{(0)}m)\\
&=& (a_{(-1)})\otimes (a_{(0)}m)\\
&=& (a \otimes_B m)_{(-1)}\otimes \alpha((a \otimes_B m)_{(0)}). 
\end{eqnarray*} So $\alpha$ is a graded $H$-colinear map.

By Lemma $\ref{lemma7}$, we get a well-defined graded $ \Bbbk $-linear map 
\begin{center}
\begin{tabular}{rl}
$ \beta :$ & $M \longrightarrow A\otimes_{B} M^{coH}$ \\ &$ m\longmapsto \theta (m_{(-1)})\otimes_{B} p_M(m_{(0)})$.
\end{tabular}
\end{center}
Let $a \in A$ and $m \in M^{coH}$. We have
\begin{eqnarray*} \beta(am)&=& \theta((am)_{(-1)}) \otimes_Bp_M((am)_{(0)})\\
&=&(\mid a_{(0)} \mid /\mid m_{(-1)} \mid) \theta(a_{(-1)}m_{(-1)}) \otimes_B p_M(a_{(0)}m_{(0)})\\
&=&(\mid a_{(0)} \mid /\mid 1_{H} \mid) \theta(a_{(-1)}1_H) \otimes_B p_M(a_{(0)}m)\\
&=&\theta(a_{(-1)}) \otimes_B \theta(S_H((a_{(0)}m)_{(-1)}))(a_{(0)}m)_{(0)}\\
&=&(\mid a_{(0)(0)}\mid / \mid m_{(-1)} \mid) \theta(a_{(-1)}) \otimes_B [\theta(S_H(a_{(0)(-1)}m_{(-1)}))a_{(0)(0)}m_{(0)}]\\
&=&\theta(a_{(-1)}) \otimes_B [\theta(S_H(a_{(0)(-1)}))a_{(0)(0)}m]\\
&=&\theta(a_{(-2)}) \otimes_B [\theta(S_H(a_{(-1)}))a_{(0)}m]\\
&=&\theta(a_{(-2)})[\theta(S_H(a_{(-1)}))a_{(0)}] \otimes_B m \quad since \quad \theta(S_H(a_{(-1)}))a_{(0)} \in B\\
&=& \theta(a_{(-2)}S_H(a_{(-1)}))a_{(0)} \otimes_B m\\
&=& \theta(1_H)(\varepsilon_H(a_{(-1)})a_{(0)}) \otimes_B m\\
&=&(1_A a) \otimes_B m=a \otimes_Bm.
\end{eqnarray*}

Thus $\beta \circ \alpha=id_{A\otimes_{B} M^{coH}}$. The proof of $\alpha\circ \beta=id_{M}$ is standard. It follows that $\alpha $ is a bijection.
\end{proof}

Theorem $\ref{theorem2}$ shows that there is a one-to-one correspondence (up to equivalence) between graded $B$-modules and colour $(A,H)$-Hopf modules.

Now we will give some consequences of Theorem $\ref{theorem2}$.

\bigskip

A colour ideal in a colour-commutative colour algebra $A$ is a (two-sided) ideal of $A$ such that $I=\bigoplus_{g \in G}(I \cap A_g)$, that is, a (two-sided) ideal of $A$ generated by homogeneous elements.

Let $A$ be a colour-commutative graded $H$-comodule colour algebra. We say that $I$ is a colour $H$-ideal of $A$ if $I$ is a colour ideal of $A$ and a graded $H$-subcomodule of $A$ with the same gradation. We say that $A$ is colour $H$-simple if the only colour $H$-ideals of $A$ are $(0)$ and $A$.

A colour division ring is a colour algebra in which every non-zero homogeneous element is invertible \cite[Chapter 2, Page 46]{NV2}. 

Let $D$ be a colour division ring and $M$ a graded $D$-module. By \cite[Section 2]{KS}, as in the ungraded case, the following assertions hold true:

(1) every $G$-graded $D$-module is graded free;

(2) any $D$-linearly independant subset of $M$ consisting of homogeneous elements can be extended to a homogeneous basis of $M$;

(3) any two homogeneous bases of $M$ over $D$ have the same cardinality.

\begin{lemma} 
\label{lemma8}
Let $A$ be a colour-commutative colour $H$-simple graded $H$-comodule colour algebra. Then $A^{coH}$ is a colour-commutative colour division ring.
\end{lemma}

\begin{proof} Let $a$ be a nonzero homogeneous element of $A^{coH}$. Then $aA$ is a colour ideal of $A$. For all $u \in A$, we have 
\begin{eqnarray*}
\chi_A(au)&=& (au)_{(-1)} \otimes (au)_{(0)}\\
&=&(\mid u_{(-1)} \mid /\mid a_{(0)} \mid)(a_{(-1)}u_{(-1)})\otimes (a_{(0)}u_{(0)})\\
&=&(\mid u_{(-1)} \mid /\mid a \mid)(1_Hu_{(-1)})\otimes (au_{(0)})\\
&=&[(\mid u_{(-1)} \mid /\mid a \mid)u_{(-1)}]\otimes (au_{(0)}) \in H \otimes (aA); 
\end{eqnarray*}
this means that $aA$ is an $H$-subcomodule of $A$. Therefore, $aA$ is a colour $H$-ideal of $A$. It contains the nonzero element $a$. So $aA$ is nonzero. Thus $aA=A$ since $A$ is colour $H$-simple. So $1_A \in aA$, and there is a homogeneous element $a' \in A$ such that $aa'=1_A$, that is, $a$ is invertible.
\end{proof}

The following corollary generalizes \cite[Theorem 7: second assertion]{Wang}.

\begin{corollary} \label{corollary4} 
Let $A$ be a colour-commutative graded $H$-comodule colour algebra and $M$ a colour $(A,H)$-Hopf module. Assume that $B=A^{coH}$ is a colour division ring (for example, if $A$ is colour $H$-simple) and there is a graded $H$-colinear map $\theta : H \rightarrow A$ which is a homomorphism of colour algebras. Then $M$ is graded free as a graded $A$-module with rank the rank of the graded $B$-module $M^{coH}$.
\end{corollary}

\begin{proof} Since $B$ is a colour-division ring, $M^{coH}$ is graded free as a graded $B$-module. So $M^{coH}$ is isomorphic to some $\bigoplus_{i \in I}B{(g_i)}$, where $I$ is an index set and $(g_i)_{i \in I}$ is a family of elements of $G$. Then we have
$$A \otimes_BM^{coH}=A \otimes_B(\bigoplus_{i \in I}B{(g_i)})=\bigoplus_{i \in I}(A \otimes_BB{(g_i)})=\bigoplus_{i \in I}(A{(g_i)}).$$ 
We deduce from Theorem $\ref{theorem2}$ that $M$ is isomorphic to $\bigoplus_{i \in I}A{(g_i)}$. This means that $M$ is graded free as a graded $A$-module with rank the rank of $M^{coH}$.
\end{proof}

\subsection{Hopf-Galois extension of graded $H$-comodule colour algebras}

Let $A$ be a graded $H$-comodule colour algebra. Set $ B= A^{coH}$. Let $M$ be a graded left $B$-module. By Lemma $\ref{lemma6}$, $A\otimes_B M $ is a colour $(A,H)$-Hopf module. Thus we get a covariant functor $$\textit{F} =A\otimes_{B} - : {}_{\textnormal{gr}-B}\mathcal{M}\longrightarrow {}_{\textnormal{gr}-A}^{\textnormal{gr}-H}\mathcal{M} $$ 
and a covariant functor $$ \textit{G}=(-)^{coH} : {}_{\textnormal{gr}-A}^{\textnormal{gr}-H}\mathcal{M}\longrightarrow {}_{\textnormal{gr}-B}\mathcal{M}; M\mapsto M^{coH}$$ 

\begin{proposition} \label{proposition1}
Let $H$ be a colour Hopf algebra and $A$ a graded $H$-comodule colour algebra. Set $B=A^{coH}$. Let $ M \in {}_{\textnormal{gr}-A}^{\textnormal{gr}-H}\mathcal{M}~~and~~N\in {}_{\textnormal{gr}-B}\mathcal{M}.$ There is a functorial isomorphism of colour $(A,H)$-Hopf modules $$ \Psi : {}_{\textnormal{gr}-A}^{\textnormal{gr}-H}{Hom(A\otimes_{B} N,M)}\rightarrow {}_{\textnormal{gr}-B}Hom(N,M^{coH}); f\mapsto [n\mapsto f(1_{A}\otimes_B n)], $$ with inverse map $ \Psi' $ given by $ g\mapsto [a\otimes_{B} n \mapsto ag(n)].$
Thus, the functors $ F $ and $ G $ form an adjoint pair with unit and counit $$\eta_{N}:N\mapsto (A\otimes_{B} N)^{coH}, n\mapsto 1_{A}\otimes_{B} n $$ and $$\alpha_{M}: A\otimes_{B}M^{coH}\rightarrow M; a\otimes_{B} m\mapsto am.$$
\end{proposition}

\begin{proof} $ f(1_{A}\otimes_{B} n)\in M^{coH}$ since $f$ is graded $H$-colinear. Let $b \in B$ and $n \in N$. We have 
\begin{eqnarray*}
\Psi (f)(bn) &=& f(1_{A} \otimes_{B} (bn)) = f((1_{A}b) \otimes_{B} n)\\
&=& f((b1_{A}) \otimes_{B} n) =  f(b(1_{A} \otimes_{B} n))\\
&=& bf(1_{A} \otimes_{B} n)= b(\Psi (f)(n)),
\end{eqnarray*}
since $f$ is graded $A$-linear, a fortiori graded $ B$-linear. Thus, the map $ \Psi $ is well defined. Clearly, $ \Psi'(g)$ is graded $ A $-linear. We show that $ \Psi'(g)$ is graded $ H$-colinear:
\begin{eqnarray*}
\Psi'(g)(a\otimes_{B} n)_{(-1)}\otimes \Psi'(g)(a\otimes_{B} n)_{(0)} &=&(ag(n))_{(-1)}\otimes (ag(n))_{(0)}\\
&=& (\mid a_{(0)} \mid /\mid g(n)_{(-1)} \mid)(a_{(-1)}g(n)_{(-1)})\otimes (a_{(0)}g(n)_{(0)})\\
&=& (\mid a_{(0)} \mid /\mid 1_H \mid)(a_{(-1)}1_H)\otimes (a_{(0)}g(n))\\
&=& a_{(-1)}\otimes a_{(0)}g(n)\\
&=& a_{(-1)}\otimes \Psi'(g)(a_{(0)}\otimes_{B} n)\\
&=&(a\otimes_{B} n)_{(-1)} \otimes \Psi'(g)((a\otimes_{B} n)_{(0)}).
\end{eqnarray*}
Thus the graded $\Bbbk$-linear map $\Psi'$ is well defined. It is easy to check that $ \Psi\circ \Psi'~~and~~ \Psi'\circ \Psi$ are respectively the identity of $ _{\textnormal{gr}-B}Hom(N,M^{coH})$ and $_{\textnormal{gr}-A}^{\textnormal{gr}-H}{Hom(A\otimes_{B} N,M)}$. This means that $\Psi$ is bijective with inverse $\Psi'$.
\end{proof}

For the proof of the following Lemma, we refer to \cite[Lemma 23]{CMZ} for the ungraded case.

\begin{lemma} \label{lemma9} Set $B= A^{coH}$. Assume that there is a graded $H$-colinear map $\theta :H\rightarrow A$ such that $\theta(1_H)=1_A$. Then for every graded left $B$-module $N$, the map $\eta: N \rightarrow (A\otimes_{B} N)^{coH}$ defined by $\eta(n)=1_A \otimes_{B} n$ is an isomorphism of graded $B$-modules.
\end{lemma}

\begin{proof} Let $\sum_i(a_i \otimes_Bn_i) \in (A\otimes_{B} N)^{coH}$. Define a map $\eta'$ from $(A\otimes_{B} N)^{coH} \rightarrow N$ by $\eta'(\sum_i(a_i \otimes_Bn_i))=\sum_i (p_A(a_i)n_i)$. Since $\sum_i(a_i \otimes_Bn_i)$ is an $H$-coinvariant element, we have
$$\chi(\sum_i(a_i \otimes_Bn_i))=\sum_i((a_i)_{-1} \otimes (a_i)_{(0)} \otimes_Bn_i)= \sum_i( 1_H \otimes a_i \otimes_Bn_i).$$ 
Let $n \in N$. We have
$$\eta'(\eta(n))=\eta'(1_A \otimes_Bn)=p_A(1_A)n= \theta({\mathcal S}_H(1_H))1_An=\theta(1_H)1_An=n.$$ 
We also have 
$$\begin{array}{rcl}\eta(\eta'(\sum_i(a_i \otimes_Bn_i)))&=& \eta(\sum_i (p_A(a_i)n_i))\\
&=&1_A \otimes_B \sum_i (p_A(a_i)n_i) \\
&=&\sum_i [1_A \otimes_B (p_A(a_i)n_i)]\\
&=& \sum_i [1_Ap_A(a_i) \otimes_B n_i \\
&=& \sum_i \theta({\mathcal S}_H((a_i)_{(-1)}))( a_i)_{(0)} \otimes_Bn_i\\
&=& \sum_i \theta({\mathcal S}_H(1_H))a_i \otimes_Bn_i\\
&=& \sum_i(a_i \otimes_Bn_i).\end{array}$$
The fourth equality is true since each $p_A(a_i) \in B$.
Thus we have proved that $\eta$ is bijective with inverse $\eta'$ 
\end{proof}

The following corollary is an application of Theorem $\ref{theorem2}$ and Lemma $\ref{lemma9}$.

\begin{corollary} \label{corollary5} Set $B= A^{coH}$. Assume that there is a graded $H$-colinear map $\theta :H\rightarrow A$ such that $\theta(1_H)=1_A$. Then the functor 
$$F= A \otimes_B- : {}_{\textnormal{gr}-B}\mathcal{M} \longrightarrow  {}_{\textnormal{gr}-A}^{\textnormal{gr}-H}\mathcal{M}$$ is Maschke, that is, every object of $_{\textnormal{gr}-A}^{\textnormal{gr}-H}\mathcal{M}$ is $F$-relative injective. If furthermore, $\theta$ is a colour algebra map, then the functor 
$$G=(-)^{coH} : {}_{\textnormal{gr}-A}^{\textnormal{gr}-H}\mathcal{M}\longrightarrow {}_{\textnormal{gr}-B}\mathcal{M}$$ is dual Maschke, that is, every object of $_{\textnormal{gr}-A}^{\textnormal{gr}-H}\mathcal{M}$ is $G$-relative projective.
\end{corollary}

\begin{proof} In \cite[Theorem 3.4]{CM}, take $\mathcal{C}={}_{\textnormal{gr}-B}\mathcal{M}$, $\mathcal{ D}={}_{\textnormal{gr}-A}^{\textnormal{gr}-H}\mathcal{M}$, $\mathcal{E}={}_{\textnormal{gr}-A}^{\textnormal{gr}-H}\mathcal{M}$ and $H=1_{\mathcal{C}}$.
By \cite[Definition 3.1]{CM}, a $F$-relative $1_{\mathcal{C}}$-injective object is called a $F$-relative injective object. By \cite[Theorem 3.4]{CM}, an object $P$ of ${}_{\textnormal{gr}-A}^{\textnormal{gr}-H}\mathcal{M}$ is $F$-relative injective if and only if the unit map $\eta_P: P \rightarrow GF(P)$ has a left inverse. The first assertion follows from Lemma $ \ref{lemma9}$ and Proposition $\ref{proposition1}$. We refer also the reader to \cite[Proposition 50]{CMZ}.	
	
By \cite[Definition 3.1]{CM}, a $G$-relative $1_{\mathcal{C}}$-projective object is called a $G$-relative projective object. By \cite[Theorem 3.4]{CM}, an object $P$ of ${}_{\textnormal{gr}-A}^{\textnormal{gr}-H}\mathcal{M}$ is $G$-relative projective if and only if the counit map $\alpha_P:(FG)(P) \rightarrow P$ has a right inverse. The result follows from Theorem $\ref{theorem2}$ and Proposition $\ref{proposition1}$. We refer also the reader to \cite[Proposition 50]{CMZ}.
\end{proof}

Set $B=A^{coH}$. We say that $A$ is a graded Hopf-Galois extension of $B$ if the graded $\Bbbk$-linear map
$$can : A \otimes_BA \rightarrow H \otimes A; \quad a \otimes_Ba' \mapsto a_{(-1)} \otimes (a_{(0)}a')$$ is an isomorphism of colour $(A,H)$-Hopf modules: $H \otimes A$ and $A \otimes_B A$ are considered as colour $(A,H)$-Hopf modules as in Lemmas $\ref{lemma4}$ and $\ref{lemma6}$.

\bigskip
We arrive to another application of Theorem $\ref{theorem2}$.
 
\begin{theorem} \label{theorem3}
Under the assumptions of Theorem $\ref{theorem2}$,
\begin{enumerate}

\item the adjoint pair $(F,G)$ between $_{\textnormal{gr}-B}\mathcal{M}$ and  $_{\textnormal{gr}-A}^{\textnormal{gr}-H}\mathcal{M}$ (Proposition $\ref{proposition1}$) is a pair of inverse equivalences.

\item $A$ is faithfully flat as a graded right $B$-module.

\item $A$ is a graded Hopf-Galois extension of $B$.
\end{enumerate}
\end{theorem}

\begin{proof} 1. By Theorem $\ref{theorem2}$, the counit map of the adjoint pair $(F,G)$ is an isomorphism of colour $(A,H)$-Hopf modules. By Lemma $\ref{lemma9}$, the unit map of the adjoint pair $(F,G)$ is an isomorphism of graded $B$-modules. We get 1.
	
2. By 1., the categories $_{\textnormal{gr}-B}\mathcal{M}$ and  $_{\textnormal{gr}-A}^{\textnormal{gr}-H}\mathcal{M}$ are equivalent. Thus if 
$$0 \rightarrow N \rightarrow N' \rightarrow N'' \rightarrow 0$$ is an exact sequence in $_{\textnormal{gr}-B}\mathcal{M}$, then 
$$0 \rightarrow F(N) \rightarrow F( N') \rightarrow F(N'') \rightarrow 0$$
is an exact sequence in $_{\textnormal{gr}-A}^{\textnormal{gr}-H}\mathcal{M}$. From this observation, $A$ is flat as a graded right $B$-module. Let   
$$0 \rightarrow N \rightarrow N' \rightarrow N'' \rightarrow 0$$ be a sequence in $_{\textnormal{gr}-B}\mathcal{M}$ such that
$$0 \rightarrow F(N) \rightarrow F( N') \rightarrow F(N'') \rightarrow 0$$ is an exact sequence in $_{\textnormal{gr}-A}\mathcal{M}$. It is easy to see that the above sequence is exact in $_{\textnormal{gr}-A}^{\textnormal{gr}-H}\mathcal{M}$. Thus applying $G$, we see that $$0 \rightarrow N \rightarrow N' \rightarrow N'' \rightarrow 0$$ is an exact sequence in $_{\textnormal{gr}-B}\mathcal{M}$. 

3. We have the following isomorphisms of graded left $B$-modules:
$$(H \otimes A)^{coH} \simeq H^{coH} \otimes A \simeq {\Bbbk}1_H \otimes A \simeq A.$$ The isomorphism is given by $a' \mapsto 1_H \otimes a'$ for all $a' \in A$. By applying $F$ and using 1., we get an isomorphism of colour $(A,H)$-Hopf modules $A \otimes_BA \simeq A \otimes_B(H \otimes A)^{coH}$: the isomorphism is given by $a \otimes_Ba' \mapsto a \otimes_B (1_H \otimes a')$ for all $a,a' \in A$. 
Since the counit map $\alpha_{H \otimes A}$ is an isomorphism of colour  $(A,H)$-Hopf modules from $A \otimes_B (H \otimes A)^{coH}$ to $H \otimes A$, we get an isomorphism of colour $(A,H)$-Hopf modules from $A \otimes_BA$ to $H \otimes A$ defined by $$a \otimes_Ba' \mapsto a(1_H \otimes a')=(a_{(-1)}1_H) \otimes (a_{(0)}a')=a_{(-1)} \otimes (a_{(0)}a').$$ 
\end{proof}

For related results to Theorem $\ref{theorem3}$ in the ungraded case, we refer to \cite[Theorem 50]{CMZ} under the assumptions that there is an $H$-colinear map $\theta$ from $H$ to $A$ such that $\theta(1_H)=1_A$ and the map $can$ is surjective.

\subsection{Graded global dimension of a colour-commutative graded $H$-comodule algebra}

Over any colour-commutative algebra $A$, every graded left $A$-module is a graded right $A$-module: $ma=(\mid m \mid / \mid a \mid)am$.

\begin{proposition} \label{proposition3} Let $H$ be a colour Hopf algebra, $A$ a colour-commutative graded $H$-comodule colour algebra, $B=A^{coH}$, $P$ and $M$ graded left $A$-modules. If there is a graded $H$-colinear map $\theta : H \rightarrow A$ which is a homomorphism of colour algebras, then 
	
\begin{enumerate}
		
\item $P \otimes_BM$ is a graded $A$-module : the action is given by $$a(p \otimes_Bm)=(\mid a_{(0)} \mid / \mid p \mid)(\theta(a_{(-1)})p) \otimes_B(a_{(0)}m)$$ for all $a \in A$, $p \in P$ and $m \in M$.
		
\item $A \otimes_BM$ is a colour $(A,H)$-Hopf module: the action is given by $$a'(a \otimes_Bm)=(\mid a'_{(0)} \mid / \mid a \mid)(\theta(a'_{(-1)})a) \otimes_B(a'_{(0)}m)$$ and the coaction is $$(a \otimes_Bm)_{(-1)} \otimes(a \otimes_Bm)_{(0)}=a_{(-1)} \otimes a_{(0)} \otimes_Bm$$ for all $a,a' \in A$ and $m \in M$.
		
\end{enumerate}
\end{proposition}

\begin{proof} 1. Let $a \in A$, $b \in B$, $p \in P$ and $m \in M$. We have 
\begin{eqnarray*} a(p \otimes_B(bm))&=& (\mid a_{(0)}\mid/ \mid p\mid) (\theta(a_{(-1)})p) \otimes_B(a_{(0)}(bm)) \\
&=& (\mid a_{(0)}\mid/ \mid p\mid) (\theta(a_{(-1)})p) \otimes_B((a_{(0)}b)m) \\
&=& (\mid a_{(0)}\mid/ \mid p \mid)(\mid a_{(0)}\mid / \mid b\mid) (\theta(a_{(-1)})p) \otimes_B((ba_{(0)})m) \\
&=& (\mid a_{(0)}\mid/ \mid p \mid)(\mid a_{(0)}\mid / \mid b\mid) (\theta(a_{(-1)})pb) \otimes_B(a_{(0)}m) \\
&=& (\mid a_{(0)}\mid/ \mid p \mid \mid b\mid) (\theta(a_{(-1)})pb) \otimes_B(a_{(0)}m) \\
&=& (\mid a_{(0)}\mid/ \mid pb \mid) (\theta(a_{(-1)})pb) \otimes_B(a_{(0)}m) \\
&=& a((pb) \otimes_B m).
\end{eqnarray*}
The third equality is true because $A$ is colour-commutative. So the $A$-action is well defined.

Let $a,a' \in A$, $p \in P$ and $m \in M$. We have
\begin{eqnarray*} (aa')(p \otimes_Bm)&=& (\mid (aa')_{(0)} \mid / \mid p \mid)\theta((aa')_{(-1)})p \otimes_B(aa')_{(0)}m\\
&=& (\mid a_{(0)}a'_{(0)}\mid / \mid p \mid)(\mid a_{(0)} \mid/ \mid a'_{(-1)} \mid)\theta(a_{(-1)}a'_{(-1)})p \otimes_B(a_{(0)}a'_{(0)})m\\
&=& (\mid a_{(0)}a'_{(0)}\mid / \mid p \mid)(\mid a_{(0)} \mid / \mid a'_{(-1)} \mid )\theta(a_{(-1)})\theta(a'_{(-1)})p \otimes_B(a_{(0)}a'_{(0)})m\\
&=&(\mid a'_{(0)}\mid / \mid p \mid)(\mid a_{(0)}\mid / \mid p\mid \mid  a'_{(-1)} \mid)[\theta(a_{(-1)})\theta(a'_{(-1)})p] \otimes_Ba_{(0)}(a'_{(0)}m)\\
&=&(\mid a'_{(0)}\mid / \mid p \mid)(\mid a_{(0)}\mid / \mid \theta(a'_{(-1)}) \mid \mid p \mid)[\theta(a_{(-1)})\theta(a'_{(-1)})p] \otimes_Ba_{(0)}(a'_{(0)}m)\\
&=&(\mid a'_{(0)}\mid / \mid p \mid)(\mid a_{(0)}\mid / \mid \theta(a'_{(-1)})p \mid)[\theta(a_{(-1)})\theta(a'_{(-1)})p] \otimes_Ba_{(0)}(a'_{(0)}m)\\
&=&(\mid a'_{(0)}\mid / \mid p \mid)a[\theta(a'_{(-1)})p \otimes_Ba'_{(0)}m]\\
&=& a[a'(p \otimes_Bm)].\end{eqnarray*}
The fourth equality is true since 
\begin{eqnarray*}(\mid a_{(0)}a'_{(0)}\mid / \mid p \mid)(\mid a_{(0)} \mid / \mid a'_{(-1)}\mid)&=&(\mid a'_{(0)}\mid / \mid p \mid)(\mid a_{(0)}\mid / \mid p \mid \mid  a'_{(-1)} \mid). \end{eqnarray*} 
The fifth equality is true since $\theta$ is homogeneous of degree $e$.
So $P \otimes_BM$ is a graded left $A$-module. 

2. By 1., $A \otimes_BM$ is a graded left $A$-module. Clearly, the $H$-coaction is well defined and $A \otimes_BM$ is a graded $H$-comodule for the given $H$-coaction. Now we have
$$ \begin{array}{l}[a'(a \otimes_Bm)]_{(-1)} \otimes [a'(a \otimes_Bm)]_{(0)} \\
=(\mid a'_{(0)} \mid/\mid a \mid)[(\theta(a'_{(-1)})a) \otimes_B(a'_{(0)}m)]_{(-1)} \otimes [(\theta(a'_{(-1)})a) \otimes_B(a'_{(0)}m)]_{(0)} \\
=(\mid a'_{(0)} \mid/\mid a \mid)(\theta(a'_{(-1)})a)_{(-1)} \otimes (\theta(a'_{(-1)})a)_{(0)} \otimes_B(a'_{(0)}m) \\
=(\mid a'_{(0)} \mid/\mid a_{(0)} \mid \mid a_{(-1)}\mid)(\mid \theta(a'_{(-1)})_{(0)} \mid/ a_{(-1)}\mid \mid)[\theta(a'_{(-1)})_{(-1)}a_{(-1)}] \otimes [\theta(a'_{(-1)})_{(0)}a_{(0)}] \otimes_B(a'_{(0)}m) \\
=(\mid a'_{(0)} \mid/ \mid a_{(0)} \mid \mid a_{(-1)}\mid)(\mid \theta(a'_{(-1)(0)}) \mid/ \mid a_{(-1)} \mid)[(a'_{(-1)(-1)})a_{(-1)}] \otimes [\theta(a'_{(-1)(0)})a_{(0)}] \otimes_B(a'_{(0)}m) \\
=(\mid a'_{(0)} \mid/ \mid a_{(0)} \mid \mid a_{(-1)}\mid)(\mid \theta(a'_{(-1)2}) \mid/ \mid a_{(-1)} \mid)[(a'_{(-1)1})a_{(-1)}] \otimes [\theta(a'_{(-1)2})a_{(0)}] \otimes_B(a'_{(0)}m) \\
=(\mid a'_{(0)} \mid/ \mid a_{(0)} \mid \mid a_{(-1)}\mid)(\mid a'_{(-1)2} \mid/ \mid a_{(-1)} \mid)[(a'_{(-1)1})a_{(-1)}] \otimes [\theta(a'_{(-1)2})a_{(0)}] \otimes_B(a'_{(0)}m) \\
=(\mid a'_{(0)} \mid / \mid a_{(0)} \mid \mid a_{(-1)}\mid)(\mid a'_{(-1)} \mid/ \mid a_{(-1)} \mid)[(a'_{(-2)})a_{(-1)}] \otimes [\theta(a'_{(-1)})a_{(0)}] \otimes_B(a'_{(0)}m). \end{array}$$
The seventh equality is true since $\theta$ is homogeneous of degree $e$.
We also have
$$ \begin{array}{l}(\mid  a'_{(0)} \mid/ \mid (a \otimes_Bm)_{(-1)} \mid)a'_{(-1)}(a \otimes_Bm)_{(-1)} \otimes a'_{(0)}(a \otimes_Bm)_{(0)} \\
=(\mid  a'_{(0)} \mid/ \mid  a_{(-1)} \mid)(a'_{(-1)}a_{(-1)}) \otimes a'_{(0)}(a_{(0)} \otimes_Bm) \\
=(\mid  a'_{(0)(0)}a'_{(0)(-1)} \mid/  \mid a_{(-1)} \mid) (\mid  a'_{(0)(0)} \mid /  \mid a_{(0)} \mid)   (a'_{(-1)}a_{(-1)}) \otimes (\theta(a'_{(0)(-1)})a_{(0)}) \otimes_B(a'_{(0)(0)}m) \\
=(\mid  a'_{(0)}a'_{(-1)} \mid/  \mid a_{(-1)} \mid) (\mid  a'_{(0)} \mid /  \mid a_{(0)} \mid)(a'_{(-2)}a_{(-1)}) \otimes (\theta(a'_{(-1)})a_{(0)}) \otimes_B(a'_{(0)}m). \end{array}$$
Since $$(\mid a'_{(0)} \mid/ \mid a_{(0)} \mid \mid a_{(-1)}\mid)(\mid a'_{(-1)} \mid/ \mid a_{(-1)} \mid)=(\mid  a'_{(0)}a'_{(-1)} \mid/  \mid a_{(-1)} \mid) (\mid  a'_{(0)} \mid /  \mid a_{(0)} \mid),$$
we have $$ \begin{array}{l}[a'(a \otimes_Bm)]_{(-1)} \otimes [a'(a \otimes_Bm)]_{(0)} \\
=(\mid  a'_{(0)} \mid/ \mid (a \otimes_Bm)_{(-1)} \mid)a'_{(-1)}(a \otimes_Bm)_{(-1)} \otimes a'_{(0)}(a \otimes_Bm)_{(0)}. \end{array}$$ 
Thus $A \otimes_BM$ is a colour $(A,H)$-Hopf module. 
\end{proof}

\bigskip

\begin{proposition} \label{proposition4} Let $H$ be a colour Hopf algebra, $A$ a colour-commutative graded $H$-comodule colour algebra, $B=A^{coH}$ and $M$ a colour $(A,H)$-Hopf module. Assume that there is a graded $H$-colinear map $\theta : H \rightarrow A$ which is a homomorphism of colour algebras and the bicharacter is symmetric. For $a\in A$ and $m \in M$, set $a \rightharpoonup m= p_M(am)$. Then 
\begin{enumerate}
\item $a \rightharpoonup p_M(m)=p_M(am)=a \rightharpoonup m$;

\item $M^{coH}$ is a graded left $A$-module : the action is given by $a \rightharpoonup m= p_M(am)$ for all $a\in A$ and $m \in M^{coH}$. In particular, $A^{coH}$ is a graded left $A$-module under the $A$-action $a' \rightharpoonup a=p_A(a'a)$ for all $a' \in A$ and $a \in A^{coH}$.
\end{enumerate}
\end{proposition}

\begin{proof} 1. We have

$$ \begin{array}{l} a\rightharpoonup p_M(m)=p_M(ap_M(m)) \\
=p_M[a\theta({\mathcal S}_{H}(m_{(-1)}))m_{(0)}] \\
=\theta({\mathcal S}_{H}\{[a\theta({\mathcal S}_{H}(m_{(-1)}))m_{(0)}]_{(-1)}\})[a\theta({\mathcal S}_{H}(m_{(-1)}))m_{(0)}]_{(0)} \\
=(\mid [a\theta({\mathcal S}_{H}(m_{(-1)}))]_{(0)} \mid / \mid m_{(0)(-1)}\mid)\theta ({\mathcal S}_{H} \{[a\theta({\mathcal S}_{H}(m_{(-1)}))]_{(-1)}m_{(0)(-1)}\})[a\theta({\mathcal S}_{H}(m_{(-1)}))]_{(0)}m_{(0)(0)} \\
=(\mid a_{(0)}\theta({\mathcal S}_{H}(m_{(-1)}))_{(0)} \mid / \mid m_{(0)(-1)}\mid)(\mid a_{(0)}\mid / \mid \theta({\mathcal S}_{H}(m_{(-1)}))_{(-1)}\mid ) \\
\theta({\mathcal S}_{H} \{a_{(-1)}\theta({\mathcal S}_{H}(m_{(-1)}))_{(-1)}m_{(0)(-1)}\})(a_{(0)}\theta({\mathcal S}_{H}(m_{(-1)}))_{(0)}m_{(0)(0)}) \\
=(\mid a_{(0)}\theta({\mathcal S}_{H}(m_{(-1)})_{(0)}) \mid / \mid m_{(0)(-1)}\mid)(\mid a_{(0)}\mid / \mid {\mathcal S}_{H}(m_{(-1)})_{(-1)}\mid) \\
\theta({\mathcal S}_{H}\{a_{(-1)}{\mathcal S}_{H}(m_{(-1)})_{(-1)}m_{(0)(-1)}\})(a_{(0)}\theta({\mathcal S}_{H}(m_{(-1)})_{(0)})m_{(0)(0)}) \\
=(\mid a_{(0)}\theta({\mathcal S}_{H}(m_{(-1)})_{2}) \mid / \mid m_{(0)(-1)}\mid)(\mid a_{(0)}\mid / \mid {\mathcal S}_{H}(m_{(-1)})_{1}\mid ) \\
\theta({\mathcal S}_{H}\{a_{(-1)}{\mathcal S}_{H}(m_{(-1)})_{1}m_{(0)(-1)}\})(a_{(0)}\theta({\mathcal S}_{H}(m_{(-1)})_{2})m_{(0)(0)}) \\
=(\mid a_{(0)}\mid \mid\theta({\mathcal S}_{H}(m_{(-1)1})) \mid / \mid m_{(0)(-1)}\mid)(\mid a_{(0)}\mid / \mid {\mathcal S}_{H}(m_{(-1)2})\mid )(\mid m_{(-1)1} \mid / \mid m_{(-1)2 \mid}) \\
\theta({\mathcal S}_{H} \{a_{(-1)}{\mathcal S}_{H}(m_{(-1)2})m_{(0)(-1)}\})(a_{(0)}\theta({\mathcal S}_{H}(m_{(-1)1}))m_{(0)(0)}) \\
=(\mid a_{(0)}\mid \mid m_{(-1)1} \mid / \mid m_{(0)(-1)}\mid)(\mid a_{(0)}\mid / \mid m_{(-1)2}\mid )(\mid m_{(-1)1} \mid / \mid m_{(-1)2} \mid) \\
\theta({\mathcal S}_{H}\{a_{(-1)}{\mathcal S}_{H}(m_{(-1)2})m_{(0)(-1)}\})(a_{(0)}\theta({\mathcal S}_{H}(m_{(-1)1}))m_{(0)(0)}) \\
=(\mid a_{(0)}\mid  \mid / \mid m_{(0)(-1)}\mid)(\mid m_{(-1)1} \mid / \mid m_{(0)(-1)})(\mid a_{(0)}\mid / \mid m_{(-1)2}\mid )(\mid m_{(-1)1} \mid / \mid m_{(-1)2} \mid) \\
\theta({\mathcal S}_{H}\{a_{(-1)}{\mathcal S}_{H}(m_{(-1)2})m_{(0)(-1)}\})(a_{(0)}\theta({\mathcal S}_{H}(m_{(-1)1}))m_{(0)(0)}) \\
=(\mid a_{(0)}\mid / \mid m_{(-1)2}\mid )(\mid a_{(0)}\mid / \mid m_{(-1)1}\mid)(\mid m_{(-2)} \mid / \mid m_{(-1)2 \mid})(\mid m_{(-2)} \mid /\mid m_{(-1)1} \mid ) \\
\theta({\mathcal S}_{H} \{(a_{(-1)}{\mathcal S}_{H}(m_{(-1)1})m_{(-1)2})\}(a_{(0)}\theta({\mathcal S}_{H}(m_{(-2)}))m_{(0)}) \\
=(\mid a_{(0)}\mid / \mid m_{(-1)}\mid)(\mid m_{(-2)} \mid /\mid m_{(-1)} \mid ) \\
\theta[{\mathcal S}_{H}\{a_{(-1)}(\varepsilon_H(m_{(-1)})1_H)\}](a_{(0)}\theta({\mathcal S}_{H}(m_{(-2)}))m_{(0)}) \\
=(\mid a_{(0)}\mid / \mid m_{(-1)}\mid)(\mid m_{(-2)} \mid /\mid m_{(-1)} \mid )\varepsilon_H(m_{(-1)}) \\
\theta({\mathcal S}_{H}(a_{(-1)}))(a_{(0)}\theta({\mathcal S}_{H}(m_{(-2)}))m_{(0)}) \\
=\theta({\mathcal S}_{H}(a_{(-1)}))(a_{(0)}\theta({\mathcal S}_{H}(m_{(-1)}))m_{(0)}) \\
=(\mid a_{(0)}\mid / \mid \theta( {\mathcal S}_{H}(m_{(-1)}))\mid)\theta({\mathcal S}_{H}(a_{(-1)}))\theta({\mathcal S}_{H}(m_{(-1)}))(a_{(0)}m_{(0)}) \\
=(\mid a_{(0)}\mid / \mid m_{(-1)}\mid)\theta({\mathcal S}_{H}(a_{(-1)}))\theta({\mathcal S}_{H}(m_{(-1)}))(a_{(0)}m_{(0)}) \end{array}$$
The ninth equality is true because $\theta$ and ${\mathcal S}_H$ are homogeneous of degree $e$. The tenth equality uses the formula mentioned before Theorem $\ref{theorem1}$.

We also have
$$\begin{array}{l} p_M(am) \\
= \theta({\mathcal S}_{H}((am)_{(-1)}))(am)_{(0)}\\
=(\mid a_{(0)} \mid / \mid m_{(-1)} \mid) \theta({\mathcal S}_{H}(a_{(-1)}m_{(-1)}))(a_{(0)}m_{(0)})\\
=(\mid a_{(0)} \mid / \mid m_{(-1)} \mid)(\mid a_{(-1)}\mid / \mid m_{(-1)}\mid) \theta[{\mathcal S}_{H}(m_{(-1)}){\mathcal S}_{H}(a_{(-1)})](a_{(0)}m_{(0)})\\
=(\mid a_{(0)} \mid / \mid m_{(-1)} \mid)(\mid a_{(-1)}\mid / \mid m_{(-1)}\mid) \theta({\mathcal S}_{H}(m_{(-1)}))\theta({\mathcal S}_{H}(a_{(-1)}))(a_{(0)}m_{(0)})\\
=(\mid a_{(0)} \mid / \mid m_{(-1)} \mid)(\mid a_{(-1)}\mid / \mid m_{(-1)}\mid)(\mid \theta({\mathcal S}_{H}(m_{(-1)}) \mid / \mid \theta({\mathcal S}_{H}(a_{(-1)})\mid)\\ \theta({\mathcal S}_{H}(a_{(-1)}))\theta({\mathcal S}_{H}(m_{(-1)}))(a_{(0)}m_{(0)})\\
=(\mid a_{(0)} \mid / \mid m_{(-1)} \mid)(\mid a_{(-1)}\mid / \mid m_{(-1)}\mid)(\mid m_{(-1)} \mid / \mid a_{(-1)}\mid) \theta({\mathcal S}_{H}(a_{(-1)}))\theta({\mathcal S}_{H}(m_{(-1)}))(a_{(0)}m_{(0)})\\
=(\mid a_{(0)} \mid / \mid m_{(-1)} \mid)\theta({\mathcal S}_{H}(a_{(-1)}))\theta({\mathcal S}_{H}(m_{(-1)}))(a_{(0)}m_{(0)})\\
\end{array}$$ 
The last equality is true since the bicharacter is symmetric.
Thus we have proved that 
$$a \rightharpoonup p_M(m)= p_M(am) \quad \forall a \in A ; m \in M.$$

2. Let $a,a' \in A$ and $m \in M$. We have
$$\begin{array}{l} a \rightharpoonup (a' \rightharpoonup m) = a \rightharpoonup( a' \rightharpoonup p_M(m))= a \rightharpoonup p_M((a' p_M(m))) \\
=p_M(a(a' p_M(m)))= p_M((aa') p_M(m))=(aa') \rightharpoonup p_M(m)) \\
=(aa') \rightharpoonup m.\end{array}$$ The first, fifth and sixth equalities are true by 1. The second and the third equalities use the definition of $p_M$. The fourth equality is true since $M$ is an $A$-module. We also have 
$1_A \rightharpoonup m=p_M(1_Am)=p_M(m)=m$ for all $m \in M^{coH}$.
\end{proof}

We arrive to the last application of Theorem $\ref{theorem2}$: it is a generalization of \cite[Theorem 9]{Wang}.  

\begin{theorem} \label{theorem4}. Let $H$ be a colour Hopf algebra, $A$ a colour-commutative graded $H$-comodule colour algebra, $B=A^{coH}$ and $M$ a colour $(A,H)$-Hopf module. Assume that $B$ is a colour division ring (for example, if $A$ is colour $H$-simple), and there is a graded $H$-colinear map $\theta : H \rightarrow A$ which is a homomorphism of colour algebras and the bicharacter is symmetric. Then 
$$\textnormal{gr}.gldim A=\textnormal{gr}.pdim_AB=pdim_AB;$$ where $\textnormal{gr}.gldim A$ and $\textnormal{gr}.pdim_A$ denote respectively the graded global dimension of $A$ and the graded projective dimension over $A$; $pdim_A$ is the projective dimension over $A$.
\end{theorem}

\begin{proof} By Proposition $\ref{proposition4}$, $B$ is a graded left $A$-module. Let us consider the projective resolution of $B$ in the category of graded left $A$-modules
$$...\rightarrow P_2 \rightarrow P_1 \rightarrow B \rightarrow 0.$$ 
Let $M$ be a graded left $A$-module. Since $B$ is a colour division ring, we have by Proposition $\ref{proposition3}$, an exact sequence of graded left $A$-modules
$$...\rightarrow P_2 \otimes_BM \rightarrow P_1 \otimes_BM \rightarrow B \otimes_BM \simeq M \rightarrow 0.$$ 
If $P$ is a graded projective right $A$-module, we know that $P$ is a direct summand in a graded free right $A$-module, that is, we have $P \oplus Q \simeq \oplus_{i \in I}A(x_i)$ as a graded right $A$-module for some graded right $A$-module $Q$; where $I$ is an index set and $(x_i, i \in I)$ is a family of elements of $G$; We deduce that
$$(P \otimes_BM) \oplus (Q  \otimes_BM) \simeq (\oplus_{i \in I}A(x_i)) \otimes_BM \simeq \oplus_{i \in I}(A(x_i) \otimes_BM) \simeq \oplus_{i \in I}(A\otimes_BM)(x_i).$$ 
By Proposition $\ref{proposition3}$, $A \otimes_BM$ is a colour $(A,H)$-Hopf module. Now $B$ is a colour division ring. So by Theorem $\ref{theorem2}$ and Corollary  $\ref{corollary4}$, $A \otimes_BM$ is projective as a graded left $A$-module. Thus each $(A \otimes_BM)(x_i)$ is a projective graded left $A$-module. It follows that $P \otimes_BM$ is a projective graded left $A$-module being a direct summand of the projective graded left $A$-module $\oplus_{i \in I}(A\otimes_BM)(x_i)$. So we have proved that the above exact sequence of graded left $A$-modules is a projective resolution of $M$ in the category of graded left $A$-modules. Thus we have $\textnormal{gr}.pdim_AM \leq \textnormal{gr}.pdim_AB$ for any graded left $A$-module $M$. It follows that $\textnormal{gr}.gldim A \leq \textnormal{gr}.pdim_AB$. It is clear that $\textnormal{gr}.pdim_AB \leq \textnormal{gr}.gldim A$. For the last equality $\textnormal{gr}.pdim_AB=pdim_AB$, see \cite[I.2.7]{NV1}.     
\end{proof}

\begin{remarks} 
\begin{enumerate}
\item In Wang's paper \cite{Wang}, $A=H$, $B=H^{coH}$ is the base field $\Bbbk$ which we can consider as a colour division ring with a trivial gradation and a trivial $H$-coaction. 
	
\item The first assertion in Proposition $\ref{proposition3}$ appears in the proof of \cite{Wang} without proof because the proof is easy since the tensor product is over $\Bbbk$. The second assertion of Proposition $\ref{proposition3}$ corresponds to Proposition 8 in Wang's paper. 

\item In Proposition $\ref{proposition4}$, we assumed that $A$ is a colour-commutative graded $H$-comodule algebra and the bicharacter is symmetric to prove that $A^{coH}$ is a graded $A$-module. This result enabled us to prove Theorem $\ref{theorem4}$. In Wang's paper, $A^{coH}=\Bbbk$ is always considered as a trivial graded $H$-module.  
\end{enumerate}
\end{remarks}

We can apply our results in the following categories.
	
$\bullet$ The category of colour Hopf modules: take $ A=H $ and $\theta =id_{H}$. If furthermore, $G=\{e\}$, then there is no gradation on $H$, we are in the category of Hopf modules.
	
$\bullet$ The category of $(A,H)$-Hopf modules: take $G=\{e\}$, so there is no gradation on $ A $ and $ H $.
	
$\bullet$ $H$ is a Hopf algebra coacting on a colour algebra $A$ in such a way that the coaction is compatible with the gradation and the product of $A$. We are in the category $_{\textnormal{gr}-A}^H{\mathcal M}$ of (gr-$A,H$)-Hopf modules, that is, the category of graded left $A$-modules, left $H$-comodules such that the $H$-coaction is compatible with the $A$-action and with the gradation.
	
$\bullet$ The category of graded $(A,H)$-Hopf modules: 
the bicharacter is trivial, that is, we are in the classical gradation.

\bigskip


\end{document}